\documentclass[11pt]{amsart}

\usepackage[utf8]{inputenc}
\usepackage[english]{babel}
\usepackage[a4paper,twoside,top=1.2in, bottom=1.1in, left=0.8in, right=0.8in]{geometry}

\usepackage{graphicx}
\usepackage{caption} 
\usepackage{amsmath}
\usepackage{amsthm}
\usepackage{amssymb}
\usepackage{esint} 
\usepackage{mathrsfs} 
\usepackage{dsfont} 
\usepackage{xcolor}
\usepackage{array}
\usepackage{hhline}
\usepackage{enumitem} 
\usepackage{comment} 
\usepackage[toc,page]{appendix} 
\usepackage{xparse} 
\usepackage{mathtools}
\usepackage{fancyhdr} 
\usepackage{ifthen} 
\usepackage{forloop} 
\usepackage{xstring}
\usepackage{emptypage} 
\usepackage{setspace}
\usepackage[initials,alphabetic]{amsrefs} 

\usepackage{tikz}
\usetikzlibrary{arrows,shapes,patterns,calc,fadings,decorations.pathreplacing,decorations.markings,decorations.pathmorphing,backgrounds}

\usepackage{hyperref} 
\hypersetup{
	colorlinks = true,
	linkcolor = {blue},
	urlcolor = {red},
	citecolor = {blue}
}

\usepackage[nameinlink,capitalise]{cleveref} 

\newcounter{results}[section] 

\theoremstyle{plain}
\newtheorem{theorem}[results]{Theorem}
\newtheorem{lemma}[results]{Lemma}
\newtheorem{proposition}[results]{Proposition}

\newtheorem*{theorem*}{Theorem}
\newtheorem*{lemma*}{Lemma}
\newtheorem*{proposition*}{Proposition}
\newtheorem*{corollary*}{Corollary}
\newtheorem*{exercise*}{Exercise}
\newtheorem*{fact*}{Fact}

\theoremstyle{remark}
\newtheorem{remark}[results]{Remark}

\newtheorem*{remark*}{Remark}
\newtheorem*{question*}{Question}

\theoremstyle{definition}
\newtheorem{definition}[results]{Definition}

\newtheorem*{definition*}{Definition}
\newtheorem*{example*}{Example}

\numberwithin{equation}{section}

\crefname{figure}{Figure}{Figures}

\ifdefined\comma 
        \renewcommand{\comma}{\ensuremath{\, \text{, }}}
\else 
        \newcommand{\comma}{\ensuremath{\, \text{, }}}
\fi

\newcommand{\N}{\ensuremath{\mathbb N}}
\newcommand{\Z}{\ensuremath{\mathbb Z}}
\newcommand{\R}{\ensuremath{\mathbb R}}

\DeclarePairedDelimiter\abs{\lvert}{\rvert} 
\DeclarePairedDelimiter\norm{\lVert}{\rVert} 
\newcommand{\scal}[2]{\ensuremath{\langle #1 , #2 \rangle}} 

\newcommand{\st}{\ensuremath{\ :\ }} 
\newcommand{\eqdef}{\ensuremath{\coloneqq}} 
\newcommand{\id}{\ensuremath{\mathrm{id}}}

\DeclareMathOperator{\tr}{tr}

\renewcommand{\d}{\ensuremath{\mathrm{d}}} 
\newcommand{\de}{\ensuremath{\, \mathrm{d}}} 

\newcommand{\grad}{\ensuremath{\nabla}} 
\newcommand{\lapl}{\ensuremath{\Delta}} 

\newcommand{\vf}{\ensuremath{\mathfrak X}} 
\DeclareMathOperator{\II}{I\!I} 
\DeclareMathOperator{\Ric}{Ric} 

\newcommand{\Haus}{\ensuremath{\mathscr H}} 
\DeclareMathOperator{\ind}{ind} 

\newcommand{\area}{\ensuremath{\Haus^2}} 
\DeclareMathOperator{\genus}{genus} 
\newcommand{\Diff}{\ensuremath{D}} 
\newcommand{\dih}{\ensuremath{\mathbb D}}

\newcommand{\vard}{\ensuremath{\boldsymbol {\mathrm{F}}}}
\newcommand{\so}{\ensuremath{\boldsymbol{\Sigma}}}

\newcommand{\smetric}{\ensuremath{\gamma}}

\newcommand{\selem}{\ensuremath{h}}

\newcommand{\limitv}{\ensuremath{\Xi}}
\newcommand{\avoids}{\ensuremath{\Theta}}

\colorlet{myGray}{gray}
\colorlet{myBlue}{blue}
\colorlet{myBlack}{black}
\colorlet{myBackground}{gray!10}

\allowdisplaybreaks

\begin{document}

\title[Equivariant index bound for min-max FBMS]{Equivariant index bound for \\ min-max free boundary minimal surfaces}

\author{Giada Franz}
     \address{ \noindent Giada Franz: 
     	\newline ETH D-Math, R\"amistrasse 101, 8092 Z\"urich, Switzerland 
     	 	\newline
     	 \textit{E-mail address: giada.franz@math.ethz.ch} 
}

\begin{abstract}
Given a three-dimensional Riemannian manifold with boundary and a finite group of orientation-preserving isometries of this manifold, we prove that the equivariant index of a free boundary minimal surface obtained via an equivariant min-max procedure à la Simon--Smith with $n$-parameters is bounded above by $n$. 
\end{abstract}

\maketitle

\thispagestyle{empty}

\setcounter{tocdepth}{1} 
\tableofcontents

\section{Introduction}

The aim of this paper is to show that, through an equivariant min-max procedure à la Simon--Smith \cite{Smith1982} and Colding--De Lellis \cite{ColdingDeLellis2003} in a three-dimensional manifold with boundary, it is possible to obtain a free boundary minimal surface with equivariant index bounded above by the number of parameters of the sweepouts.

The min-max theory developed by Simon--Smith and Colding--De Lellis differs from the one by Almgren--Pitts \cite{Almgren1965}, \cite{Pitts1981} and Marques--Neves \cite{MarquesNeves2014}, \cite{MarquesNeves2017} since, in the former one, stronger regularity and convergence conditions are imposed on the sweepouts (cf. \cite{MarquesNeves2014}*{Section 2.11}) and the ambient dimension is assumed to be equal to three. 
This makes the Simon--Smith approach less flexible; indeed, note for example that the sweepout constructed by Marques and Neves in \cite{MarquesNeves2014} to prove the Willmore conjecture does not satisfy the Simon--Smith assumptions.

On the other hand, the stronger regularity makes proofs easier and enables to obtain a partial control on the topology of the resulting surface. This has been an advantage of the Simon--Smith approach when one wants to construct a minimal surface with a certain topological type. Indeed, the only way we know so far to obtain some information on the topology of the resulting minimal hypersurface in the Almgren--Pitts min-max theory is to first control the index and then use that the index bounds the topology (see for example \cite{AmbCarSha18-Index},
\cite{Maximo2018}, \cite{Song2020} and \cite{AmbCarSha18-IndexFBMS} for the free boundary case), although such bounds (even when effective) are far from sharp.

Luckily, many proofs working in the Almgren--Pitts setting can be adapted to the Simon--Smith one. This is also the case for the upper bound on the index, proved in \cite{MarquesNeves2016} for the Almgren--Pitts min-max theory (see also \cite{MarquesNeves2021}, and \cite{GuangLiZhou2021} for the free boundary setting).
Indeed, this paper consists in adapting those arguments to our ``smoother'' setting (see \cite{MarquesNeves2016}*{Section 1.3}), but with the additional presence of a group of isometries imposed to the min-max procedure.
What we obtain is an upper bound on the \emph{equivariant index} of the surface resulting from this \emph{equivariant min-max procedure}.
Let us now describe precisely the statement of this result. To this purpose, we first give some preliminary definitions.

\subsection{Setting and definitions} 
The equivariant min-max theory considered in this paper has been mostly developed by Ketover in \cite{Ketover2016equiv} for the closed case and in \cite{Ketover2016fb} for the free boundary one, as proposed by Pitts--Rubinstein in \cite{PittsRubinstein1988}. One of its primary goals is to construct new families of minimal surfaces in $S^3$ and of free boundary minimal surfaces in $B^3$. Indeed, encoding the right symmetry group in the min-max procedure allows us to produce surfaces with fully controlled topology, as done by the author in joint work with Carlotto and Schulz in \cite{CarlottoFranzSchulz2020}.

Our setting is a three-dimensional compact Riemannian manifold $(M^3,\smetric)$ with strictly mean convex boundary and with a \emph{finite} group $G$ of orientation-preserving isometries of $M$.

\begin{remark}
We assume the manifold $M$ to have strictly mean convex boundary because we want it to satisfy the following property.
\begin{description}
\item [$(\mathfrak{P})$]\label{HypP} \hspace{1.2ex} If $\Sigma^2\subset M$ is a smooth, connected, complete (possibly noncompact), embedded surface with zero mean curvature which meets the boundary of the ambient manifold orthogonally along its own boundary, then $\partial\Sigma = \Sigma\cap\partial M$. 
\end{description}
Here we ask for more, namely, the strict mean convexity of the boundary of the ambient manifold (which implies \hyperref[HypP]{$(\mathfrak{P})$}), because we want it to be an open property with respect to the class of smooth metrics. 
\end{remark}

\begin{remark} \label{rem:NotProperInLi}
Li in \cite{Li2015} studies the free boundary min-max problem without curvature assumptions on the ambient manifold $M$. However, as remarked in \cite{CarlottoFranz2020}*{Appendix A} (see also \cite{CarlottoFranz2020}*{pp. 5}), there are problems in defining the Morse index for free boundary minimal surfaces that are not properly embedded in $M$. Moreover, the compactness results in \cite{AmbCarSha18Compactness}, which we use for example in \cref{sec:BumpyIsGGeneric}, need property \hyperref[HypP]{$(\mathfrak{P})$}.
That is another good reason why we require \hyperref[HypP]{$(\mathfrak{P})$} to hold in this paper.
\end{remark}

\begin{remark}
Here we assume the group $G$ to be finite, but it is due mentioning that the equivariant min-max theory has been developed also in the case of a compact \emph{connected} Lie group with cohomogeneity not $0$ or $2$ in \cite{Liu2021} and in the Almgren--Pitts setting with a compact Lie group of cohomogeneity greater or equal than $3$ in \cite{Wang2020}.
\end{remark}

In what follows, we denote by $I^n=[0,1]^n$ the product of $n$-copies of the unit interval. Moreover, when we say that $\Sigma^m$ is a \emph{hypersurface} (or \emph{surface} if $m=2$) in an ambient manifold $M^{m+1}$, we mean that $\Sigma$ is smooth, complete and properly embedded, i.e., $\partial\Sigma=\Sigma\cap\partial M$.

\begin{definition} \label{def:GSweepout}
A family $\{\Sigma_t\}_{t\in I^n}$ of subsets of a three-dimensional Riemannian manifold $M$ is said to be a \emph{generalized family} of surfaces if there are a finite subset $T$ of $I^n$ and a finite set of points $P$ of $M$ such that:
\begin{enumerate}[label={\normalfont(\roman*)}]
    \item $t\mapsto\Sigma_t$ is continuous in the sense of varifolds;
    \item $\Sigma_t$ is a surface for every $t\not\in T\cup \partial I^n$;
    \item for $t\in T\setminus\partial I^n$, $\Sigma_t$ is a surface in $M\setminus P$.
\end{enumerate}

Moreover, the generalized family $\{\Sigma_t\}_{t\in I^n}$ is said to be \emph{smooth} if it holds also that:
\begin{enumerate}[resume*]
    \item $t\mapsto\Sigma_t$ is smooth for $t\not\in T\cup\partial I^n$;
    \item for $t\in T\setminus \partial I^n$, $\Sigma_s\to\Sigma_t$ smoothly in $M\setminus P$ as $s\to t$.
\end{enumerate}

Finally, we say that $\{\Sigma_t\}_{t\in I^n}$ is a \emph{$G$-sweepout} if it is a smooth generalized family of surfaces and $\Sigma_t$ is $G$-equivariant for all $t\in I^n$, i.e., $h(\Sigma_t) = \Sigma_t$ for all $h\in G$ and $t\in I^n$.
\end{definition}

\begin{remark}
Slightly different variations of this definition are given in several references (see for example \cite{ColdingDeLellis2003}*{Definition 1.2}, \cite{DeLellisPellandini2010}*{Definition 0.5}, \cite{ColdingGabaiKetover2018}*{pp. 2836}). 
In order to have the regularity and the index bound on the resulting surface, it is sufficient to consider generalized families of surfaces, however smoothness is needed for the genus bound (see \cite{DeLellisPellandini2010}*{Definition 0.5}).
\end{remark}

\begin{definition} \label{def:GEquivIsotopy}
We say that a smooth map $\Phi:I^n\times M\to M$ is a \emph{$G$-equivariant isotopy} if $\Phi_t\eqdef \Phi(t,\cdot) \in \operatorname{Diff}_G(M)$ for all $t\in I^n$, namely $\Phi_t\colon M\to M$ is a diffeomorphism such that $\Phi_t\circ h = \Phi_t$ for all $h\in G$, $t\in I^n$.
\end{definition}

\begin{definition}
Given a $G$-sweepout $\{\Sigma_t\}_{t\in I^n}$, we define its \emph{$G$-saturation} as
\[
\Pi \eqdef \{ \{\Phi_t(\Sigma_t)\}_{t\in I^n} \st \text{$\Phi\colon I^n\times M\to M$ is a $G$-equivariant isotopy with $\Phi_t=\id$ for $t\in \partial I^n$} \}.
\]
Then the \emph{min-max width} of $\Pi$ is defined as
\[
W_\Pi \eqdef \inf_{\{\Lambda_t\}\in \Pi} \sup_{t\in I^n} \area(\Lambda_t).
\]
\end{definition}
\begin{remark}
Note that here we consider isotopies $\Phi$ such that $\Phi_t\colon M\to M$ is a diffeomorphism of the ambient manifold. This choice does not coincide with the definition in \cite{Li2015}, where \emph{outer isotopies} are considered (see also \cref{rem:NotProperInLi}).
\end{remark}
\begin{remark}
Note that the uniform bound on the ``bad points'' (the points belonging to $P$ in \cref{def:GSweepout}), which is required in \cite{ColdingDeLellis2003}*{Remark 1.3} (see also \cite{DeLellisPellandini2010}*{Remark 0.2}), is trivially satisfied for the $G$-saturation of a $G$-sweepout.
\end{remark}

\subsection{Main result}
We obtain the following equivariant min-max theorem, which is the combination of regularity, genus bound and equivariant index bound, the latter being the main object of this paper. Indeed, the regularity and the genus bound were addressed in \cite{Ketover2016fb}*{Theorem 3.2} (see also \cite{Ketover2016equiv}*{Theorem 1.3}, where Ketover dealt with the closed case, for which many arguments are similar).

\begin{theorem} \label{thm:EquivMinMax}
Let $(M^3, \smetric)$ be a three-dimensional Riemannian manifold with strictly mean convex boundary and let $G$ be a finite group of orientation-preserving isometries of $M$.
Let $\{\Sigma_t\}_{t\in I^n}$ be a $G$-sweepout and let $\Pi$ be its $G$-saturation. Assume that
\[
W_\Pi > \sup_{t\in \partial I^n} \area(\Sigma_t).
\]
Then there exist a minimizing sequence $\{\{\Sigma^j_t\}_{t\in I^n}\}_{j\in\N}\subset \Pi$ (i.e., $\lim_{j\to\infty} \sup_{t\in I^n} \area(\Sigma_t^j) = W_\Pi$) and a sequence $\{t_j\}_{j\in\N}\subset I^n$ such that $\{\Sigma^j = \Sigma^j_{t_j}\}_{j\in\N}$ is a min-max sequence (i.e., $\lim_{j\to\infty} \area(\Sigma^j) = W_\Pi$) converging in the sense of varifolds to $\limitv\eqdef \sum_{i=1}^k m_i\limitv_i$, where $\limitv_i$ are disjoint free boundary minimal surfaces and $m_i$ are positive integers.

Moreover, the $G$-equivariant index of the support $\operatorname{spt}(\limitv) = \bigcup_{i=1}^k \limitv_i$ of $\limitv$ is less or equal than $n$, namely
\[
\ind_G(\operatorname{spt}(\Xi)) = \sum_{i=1}^k \ind_G(\Xi_i) \le n,
\]
and the following genus bound holds
\[
\sum_{i\in\mathcal O} m_i\genus(\limitv_i) + \sum_{i\in \mathcal N} m_i(\genus(\limitv_i)-1) \le \liminf_{j\to+\infty}\liminf_{\tau\to t_j} \genus(\Sigma^j_{\tau}),
\]
where $\mathcal O$ is the set of indices $i\in\{1,\ldots,k\}$ such that $\Xi_i$ is orientable and $\mathcal N$ is the set of indices $i\in\{1,\ldots,k\}$ such that $\Xi_i$ is nonorientable.
\end{theorem}

The precise definition of equivariant index is given in \cref{sec:EquivSpectrum}, but it is just what one would expect, i.e., the maximal dimension of a linear subspace of the \emph{$G$-equivariant vector fields} where the second variation of the area functional is negative definite.

\begin{remark}
The assumptions that $M$ has dimension $3$ and that $G$ consists of \emph{orientation-preserving} isometries are required in \cite{Ketover2016equiv} and \cite{Ketover2016fb}, which we use for the regularity and genus bound in \cref{prop:AlmostMinSeq} (see also \cref{rem:SymmetriesLocally}).
In this paper, such assumptions are used only in \cref{sec:ConvBoundGIndex}. Elsewhere the arguments work for any finite group of isometries $G$ and any ambient dimension $3\le m+1\le 7$.
\end{remark}

\begin{remark}
The equivariant index estimate in \cref{thm:EquivMinMax} implies the analogous result in the closed case, since, if we repeat the proofs forgetting about the boundary, everything works in the same way (actually more easily). 
\end{remark}

\subsection{Applications}

As a direct application, we deduce that the family of free boundary minimal surfaces in $B^3$ constructed in \cite{CarlottoFranzSchulz2020} have dihedral equivariant index \emph{equal to} one. 

Let us briefly recall the geometry of these surfaces in order to better appreciate the theorem below.
Fixed any $g\ge 1$, recall that the dihedral group $\mathbb{D}_{g+1}$ acting on $\overline{B^3}$ is defined as the subgroup of Euclidean isometries generated by the rotation of angle $2\pi/(g+1)$ around the vertical axis $\xi_0\eqdef \{(0,0,r) \st r\in[-1,1]\}$ and by the rotations of angle $\pi$ around the $g+1$ horizontal axes $\xi_k \eqdef \{ (r\cos(k\pi/(g+1)), r\sin(k\pi/(g+1)), 0) \st r\in [-1,1] \}$ (see also \cite{CarlottoFranzSchulz2020}*{pp. 1}).
Then the free boundary minimal surface $M_g\subset B^3$ constructed in \cite{CarlottoFranzSchulz2020}*{Theorem 1.1} is $\dih_{g+1}$-equivariant, has genus $g$ and connected boundary.
Moreover, $M_g$ contains the horizontal axes $\xi_1,\ldots,\xi_{g+1}$ and intersects $\xi_0$ orthogonally.

\begin{theorem} \label{thm:Dg1Index}
For all $g\ge 1$, we can assume that the free boundary minimal surface $M_g\subset B^3$ constructed in \cite{CarlottoFranzSchulz2020}*{Theorem 1.1} has $\dih_{g+1}$-equivariant index equal to $1$.
\end{theorem}

\begin{remark}
Note that the phrase ``we can assume'' in the statement is due to the fact that in \cref{thm:EquivMinMax} we prove the existence of a surface with a bound on the equivariant index, but we do not show that every surface obtained from the min-max procedure has this property. Therefore, \cref{thm:EquivMinMax} has to be applied in place of the min-max theorem used at the beginning of Section 4 in \cite{CarlottoFranzSchulz2020}.
\end{remark}

\begin{remark}
Even if, for simplicity, we state the result only for the surfaces $M_g$ constructed in \cite{CarlottoFranzSchulz2020}*{Theorem 1.1}, \cref{thm:EquivMinMax} can be applied to any surface obtained via an equivariant min-max procedure fitting the framework defined above. See \cref{rem:OtherFamilies} for further details on this matter.
\end{remark}

Besides its own interest, \cref{thm:Dg1Index} is meant to be a first step toward the computation of the (nonequivariant) Morse index of the family of surfaces $M_g$ (or possibly any another family arising from an equivariant min-max procedure, even with multiple parameters in play). 
This would be particularly relevant, as the only free boundary minimal surfaces in $B^3$ for which we know the index are the equatorial disc (with index $1$) and the critical catenoid (with index $4$, see \cite{Devyver2019}, \cite{SmithZhou2019}, \cite{Tran2020}).
The idea is that the equivariant index gives information on the index of any of the isometric parts (which are $2(g+1)$ in the case of $M_g$) of the equivariant surface. 
This problem will be addressed elsewhere.

\vspace{3ex}
\textbf{Acknowledgements.} 
The author wishes to thank her supervisor Alessandro Carlotto for many useful comments and discussions, and Tongrui Wang for helpful observations on the first version of this manuscript.
This project has received funding from the European Research Council (ERC) under the European
Union’s Horizon 2020 research and innovation programme (grant agreement No. 947923).

\section{Basic notation}
Let $(M^{m+1},\gamma)$ be a compact Riemannian manifold with boundary and let $\Sigma^m\subset M$ be a hypersurface.
Then we denote by:
\begin{itemize}
\item $\vf(M)$ the set of vector fields on $M$ \emph{tangent} to $\partial M$, namely $X\in\vf(M)$ if and only if $X(x)\in T_x\partial M$ for all $x\in \partial M$. In what follows, we consider (tangent) vector fields in $\vf(M)$ often omitting the adjective ``tangent''.
\item $X^\perp \in \Gamma(N\Sigma)$ the normal component to $\Sigma$ of a vector $X\in\vf(M)$, where $\Gamma(N\Sigma)$ denotes the sections of the normal bundle of $\Sigma$.
\item $D$ the connection on $M$, $\grad$ the induced connection on $\Sigma$ and $\grad^\perp$ the induced connection on the normal bundle of $\Sigma$.
\item $\Ric_M$ the Ricci curvature of $M$.
\item $\eta$ the outward unit conormal vector field to $\partial\Sigma$.
\item $\hat\eta$ the outward unit conormal vector field to $\partial M$ (which coincides with $\eta$ along $\partial\Sigma$ when $\Sigma$ satisfies the free boundary property).
\item $\II^{\partial M} = \gamma(D_XY, \hat \eta)$ the second fundamental form of $\partial M\subset M$.
\item $A(X,Y) = (D_XY)^\perp$ the second fundamental form of $\Sigma\subset M$ and $\abs{A}^2$ its Hilbert-Schmidt norm.
\end{itemize}

\section{Equivariant spectrum} \label{sec:EquivSpectrum}

In this section $(M^{m+1},\smetric)$ denotes a compact Riemannian manifold with boundary and $G$ denotes a finite group of isometries of $M$.

\begin{definition}

A (tangent) vector field $X\in\vf(M)$ is $G$-equivariant if $\selem_*X = X$ for all $\selem\in G$. We denote by $\vf_G(M)$ the set of all $G$-equivariant vector fields.
\end{definition}

\begin{definition}
A varifold $V$ in $M$ is $G$-equivariant if $\selem_*V=V$ for all $\selem\in G$.
We denote by $\mathcal{V}^2_G(M)$ the set of $2$-dimensional $G$-equivariant varifolds supported in $M$, endowed with the weak topology. Recall that $\mathcal{V}^2_G(M)$ is metrizable and we denote by $\vard$ a metric metrizing it (see \cite{Pitts1981}*{pp. 66} or \cite{MarquesNeves2014}*{pp. 703}). Moreover, given a varifold $V$ in $M$, we denote by $\norm{V}$ the Radon measure in $M$ associated with $V$.
\end{definition}

\begin{remark} \label{rem:GequivVecFieldToFlow}
Given a $G$-equivariant vector field $X\in\vf_G(M)$, let $\Phi\colon [0,+\infty)\times M\to M$ be the associated flow. Then, if $V$ is a $G$-equivariant varifold, $(\Phi_t)_*V$ is $G$-equivariant as well for all $t\in[0,+\infty)$. Indeed, $\Phi\colon[0,+\infty)\times M\to M$ is a $G$-equivariant isotopy (see \cref{def:GEquivIsotopy}).
\end{remark}

\begin{definition}
Given a varifold $V$ in $M$ and a vector field $X\in\vf(M)$, we denote by $\delta V(X)$ the first variation of the area of $V$ along $X$, namely
\[
\delta V(X) = \frac{\d}{\d t}\Big|_{t=0} \norm{(\Phi_t)_*V}(M),
\]
where $\Phi\colon[0,+\infty)\times M\to M$ is the flow generated by $X$.
If $V\in\mathcal{V}^2_G(M)$ is $G$-equivariant, we say that $V$ is $G$-stationary if $\delta V(X)=0$ for all $G$-equivariant vector fields $X\in \vf_G(M)$.
\end{definition}

\begin{remark} \label{rem:stationaryGstationary}
By the principle of symmetric criticality by Palais (see \cite{Palais1979}, or \cite{Ketover2016equiv}*{Lemma 3.8} for the result in this setting), we have that $V\in\mathcal{V}^2_G(M)$ is $G$-stationary if and only if it is stationary.
\end{remark}

Now, let $\Sigma^m\subset M$ be a compact $G$-equivariant hypersurface and let $\Phi\colon[0,+\infty)\times M\to M$ be the flow associated to a $G$-equivariant vector field $X\in\vf_G(M)$, then we have that
\[
\frac{\d}{\d t}\Big|_{t=0} \Haus^m(\Phi_t(\Sigma)) = - \int_\Sigma \scal{H}{X} \de \Haus^{m} + \int_{\partial \Sigma} \scal{\eta}{X} \de \Haus^{m-1} = 0,
\]
where $\eta$ is the outward unit conormal vector field to $\partial \Sigma$.
Moreover, if $\Sigma$ is stationary (namely a \emph{free boundary minimal hypersurface}), we have
\begin{align*}
\frac{\d^2}{\d t^2}\Big|_{t=0} \Haus^m(\Phi_t(\Sigma)) &= Q^\Sigma(X^\perp, X^\perp),
\end{align*}
where 
\[
Q^\Sigma(X^\perp, X^\perp)\eqdef \int_\Sigma (\abs{\grad^\perp X^\perp}^2 - (\Ric_M(X^\perp,X^\perp) + \abs{A}^2\abs{X^\perp}^2)) \de \Haus^m + \int_{\partial\Sigma} \II^{\partial M}(X^\perp, X^\perp) \de \Haus^{m-1}.
\]

Let $\Gamma_G(N\Sigma)$ denote the sections of the normal bundle of $\Sigma$ obtained as restriction to $\Sigma$ of $G$-equivariant vector fields in $M$.
Then, the \emph{$G$-equivariant (Morse) index} $\ind_G(\Sigma)$ of $\Sigma$ is defined as the maximal dimension of a linear subspace of $\Gamma_G(N\Sigma)$ where $Q^\Sigma$ is negative definite.

\begin{remark} \label{rem:GequivExtension}
Note that, given a $G$-equivariant submanifold $\Sigma$, any $G$-equivariant vector field defined along $\Sigma$ can be extended to a $G$-equivariant vector field on the ambient manifold $M$. Indeed, we can first extend it to a vector field $\tilde X\in\vf(M)$ (not necessarily $G$-equivariant) and then define \[X=\frac{1}{\abs{G}} \sum_{\selem\in G} \selem_*\tilde X.\]
\end{remark}

\subsection{Equivariant index for two-sided hypersurfaces} \label{sec:twoside}

If $\Sigma^m\subset M^{m+1}$ is two-sided, given a $G$-equivariant section $Y\in\Gamma_G(N\Sigma)$ of the normal bundle, we can write it as $Y=u\nu$, where $u\in C^\infty(\Sigma)$ and $\nu$ is a choice of unit normal to $\Sigma$.
Then, the second variation of the volume of $\Sigma$ along $Y$ is given by
\begin{align*}
Q^\Sigma(Y,Y)=Q_\Sigma(u,u) &\eqdef \int_\Sigma (\abs{\grad u}^2 - (\Ric_M(\nu,\nu)+\abs{A}^2)u^2) \de \Haus^m + \int_{\partial \Sigma} \II^{\partial M}(\nu,\nu) u^2 \de \Haus^{m-1}\\
&= -\int_\Sigma u \mathcal L_\Sigma u \de \Haus^m + \int_{\partial \Sigma}(u\partial_\eta u + \II^{\partial M}(\nu,\nu) u^2 ) \de \Haus^{m-1},
\end{align*}
where $\mathcal L_\Sigma \eqdef \lapl + \Ric_M(\nu,\nu) + \abs{A}^2$ is the Jacobi operator associated to $\Sigma$.

Now let us further assume that $\Sigma$ is connected. Note that, if $Y=u\nu$ is $G$-equivariant, for all $h\in G$ we must have
\[
\selem_*(u\nu) = u\nu \implies u(\selem(x))\nu(\selem(x)) =  \d \selem_x[u(x)\nu(x)] = u(x) \d \selem_x[\nu(x)] = \operatorname{sgn}_\Sigma(\selem) u(x) \nu(\selem(x)),
\]
where $\operatorname{sgn}_\Sigma(\selem) = 1$ if $\selem_*\nu = \nu$ and $\operatorname{sgn}_\Sigma(\selem)=-1$ if $\selem_*\nu = -\nu$. Indeed, $h_*\nu $ is equal to $\nu$ or $-\nu$, because $h$ is an isometry and $h(\Sigma)=\Sigma$ (hence $h_*(N\Sigma)=N\Sigma$), and $\Sigma$ is connected (thus the sign does not depend on the point on $\Sigma$). 
Note that $\operatorname{sgn}_\Sigma(\selem_1\circ\selem_2) = \operatorname{sgn}_\Sigma(\selem_1)\operatorname{sgn}_\Sigma(\selem_2)$.
This leads to the following definition.

\begin{definition} \label{def:EquivFuncSpaces}
Given a compact connected Riemannian manifold $\Sigma^m$, a finite group $G$ of isometries of $\Sigma$ and a multiplicative function $\operatorname{sgn}_\Sigma\colon G\to \{-1,1\}$, we define the spaces
\begin{align*}
C^\infty_G(\Sigma) &\eqdef \{u\in C^\infty(\Sigma) \st u\circ \selem = \operatorname{sgn}_\Sigma(\selem) u\ \ \forall \selem\in G\}, \\
L^2_G(\Sigma) &\eqdef \overline{C^\infty_G(\Sigma)}^{\norm{\cdot}_{L^2}}, \\
H^1_G(\Sigma) &\eqdef \overline{C^\infty_G(\Sigma)}^{\norm{\cdot}_{H^1}}.
\end{align*}
\end{definition}

\begin{remark}
Observe that $L^2_G(\Sigma)$ is a Hilbert space endowed with the scalar product $(u,v)_{L^2}=\int_\Sigma uv \de \Haus^m$ and $C^\infty_G(\Sigma)$ is dense in $L^2_G(\Sigma)$.
\end{remark}

The $G$-equivariant (Morse) index of $\Sigma$ coincides with the maximal dimension of a subspace of $C^\infty_G(M)$ where $Q_\Sigma$ is negative definite.
Moreover, thanks to \cref{thm:DiscreteSpectrum}, the elliptic problem
\[
\begin{cases}
-\mathcal L_\Sigma\varphi = \lambda \varphi & \text{in $\Sigma$}\\
\partial_\eta \varphi = -   \II^{\partial M}(\nu,\nu)\varphi & \text{in $\partial\Sigma$}
\end{cases}
\]
admits a discrete spectrum $\lambda_1\le \lambda_2\le \ldots\le \lambda_k\le \ldots\to+\infty$ with associated $L^2_G(\Sigma)$-orthonormal basis of eigenfunctions $(\varphi_k)_{k\ge 1}\subset C^\infty_G(\Sigma)$ of $L^2_G(\Sigma)$. The $G$-equivariant index equals the number of negative eigenvalues.

\begin{remark}
Note that the assumption of connectedness of $\Sigma$ is not restrictive. Indeed, if $\Sigma$ consists of several connected components $\Sigma_1,\ldots,\Sigma_k$ for some $k\ge 2$, then we can consider each connected component separately and $\ind_G(\Sigma) = \sum_{i=1}^k\ind_G(\Sigma_i)$.
\end{remark}

\subsection{Equivariant index for one-sided hypersurfaces}
If $\Sigma^m\subset M^{m+1}$ is one-sided, the elliptic problem
\begin{equation}\label{eq:OneSidedProblem}
\begin{cases}
-\lapl_\Sigma^\perp Y -\Ric_M^\perp(Y,\cdot) - \abs{A}^2Y = \lambda Y & \text{in $\Sigma$}\\
\grad^\perp_\eta Y = -   (\II^{\partial M}(Y,\cdot))^\sharp & \text{in $\partial\Sigma$},
\end{cases}
\end{equation}
on the $G$-equivariant sections $\Gamma_G(N\Sigma)$ of the normal bundle, admits a discrete spectrum $\lambda_1\le \lambda_2\le \ldots\le \lambda_k\le \ldots\to+\infty$ and the $G$-equivariant index coincides with the number of negative eigenvalues. This follows from the two-sided case applied to the double cover of $\Sigma$.
Indeed, let us consider the double cover $\pi\colon\tilde\Sigma\to\Sigma$, then every isometry $\selem\in G$ lifts to two isometries $\tilde \selem_1, \tilde \selem_2$ of $\tilde\Sigma$, where $\tilde \selem_1\circ\tilde \selem_2^{-1}$ is given by the isometric involution $i\colon\tilde\Sigma\to\tilde\Sigma$ associated to the universal cover. We denote by $\tilde G$ the finite group generated by these isometries.
Every $Y\in\Gamma_G(N\Sigma)$ lifts to a vector field $\tilde Y\in\Gamma(N\tilde\Sigma)$ and we can write it as $\tilde Y= u\tilde\nu$, where $\tilde \nu$ is a global unit normal to $\tilde\Sigma$ and $u\in C^\infty_{\tilde G}(\tilde\Sigma)$.
Vice versa, for all $u\in C^\infty_{\tilde G}(\tilde\Sigma)$, the vector field $Y\eqdef \pi_*(u\tilde\nu)$ is well-defined and belongs to $\Gamma_G(N\Sigma)$. Hence we can just apply \cref{thm:DiscreteSpectrum} to $\tilde\Sigma$ and $\tilde G$ (as we did in the previous subsection) to obtain the desired properties on the spectrum of problem \eqref{eq:OneSidedProblem}.

\section{Free boundary minimal surfaces with bounded equivariant index} \label{sec:ConvBoundGIndex}

In this section, roughly speaking, we prove that a limit of free boundary minimal surfaces with bounded equivariant index satisfies the same bound on the equivariant index.

Throughout the section, we assume $(M^3,\gamma)$ to be a three-dimensional compact Riemannian manifold with boundary satisfying property \hyperref[HypP]{$(\mathfrak{P})$}, $G$ to be a finite group of \emph{orientation-preserving} isometries of $M$ and $\mathcal{S}\subset M$ to be the singular locus of $G$, namely
\[
\mathcal{S} \eqdef \{x\in M \st \exists\, \selem\in G,\ \selem(x) = x\}.
\]
Moreover, given a free boundary minimal surface $\Sigma^2\subset M$ and an open subset $U\subset M$, we let $\mu_1(\Sigma\cap U)$ denote the first eigenvalue of the problem
\begin{equation} \label{eq:LocalEigenvProb}
\begin{cases}
-\lapl_{\Sigma}^\perp Y -\Ric_M^\perp(Y,\cdot) - \abs{A}^2Y = \mu Y & \text{in $\Sigma\cap U$}\\
\grad^\perp_\eta Y = -   (\II^{\partial M}(Y,\cdot))^\sharp & \text{in $\partial\Sigma \cap U$}\\
Y=0 & \text{in $\partial(\Sigma\cap U)\setminus \partial \Sigma$}
\end{cases}
\end{equation}
on the sections $\Gamma(N(\Sigma\cap U))$ of the normal bundle. Furthermore, if $\Sigma$ and $\Sigma\cap U$ are $G$-equivariant, let $\lambda_1(\Sigma\cap U)$ be the first eigenvalue of the same problem \eqref{eq:LocalEigenvProb} on the \emph{$G$-equivariant} sections $\Gamma_G(N(\Sigma\cap U))$ of the normal bundle (see \cref{sec:EquivSpectrum}).

\begin{remark}
We stress the fact that we use the letters $\lambda$ and $\mu$ to denote the eigenvalues respectively \emph{taking} and \emph{not taking} into account the symmetry group.
\end{remark}

\begin{lemma}[{cf. \cite{Ketover2016equiv}*{Proposition 4.6}}] \label{lem:FirstEigenvalues}
In the setting above, let $\Sigma^2\subset M$ be a $G$-equivariant free boundary minimal surface in $M$ and fix $x\in\Sigma\setminus\mathcal{S}$.  
Then, there exists $\varepsilon_0>0$ such that, for all $0<\varepsilon<\varepsilon_0$, the $G$-equivariant subset $U_\varepsilon(x)\eqdef \bigcup_{\selem\in G} \selem(B_\varepsilon(x))$ consists of exactly $\abs{G}$ disjoint balls $\{B_\varepsilon(\selem(x))\}_{\selem\in G}$, and it holds
\[
\mu_1(\Sigma\cap B_\varepsilon(x)) = \mu_1(\Sigma\cap U_\varepsilon(x)) = \lambda_1(\Sigma\cap U_\varepsilon(x)).
\]
\end{lemma}
\begin{proof}
First, note that $B_\varepsilon(\selem(x)) = \selem(B_\varepsilon(x))$ since every $\selem\in G$ is an isometry. Moreover, if $\varepsilon_0$ is sufficiently small (in particular smaller than half of the distance between $x$ and $\selem(x)$ for all $\selem\in G$),
then we have that $\selem(B_\varepsilon(x)) \cap B_\varepsilon(x) = B_\varepsilon(h(x)) \cap B_\varepsilon(x) \not= \emptyset$ for some $\selem\in G$ if and only if 
$h(x)=x$, which implies that $h=\id$ because $x\not\in\mathcal{S}$.
In particular, we get that the subset $U_\varepsilon(x)\eqdef \bigcup_{\selem\in G} \selem(B_\varepsilon(x))$ consists of exactly $\abs{G}$ disjoint balls.

Now, we let $\varepsilon>0$ be sufficiently small such that $\Sigma\cap B_\varepsilon(x)$ is two-sided. Then, we can write any section $X\in \Gamma(N(\Sigma\cap B_\varepsilon(x))$ of the normal bundle as $X = u\nu$, where $u\in C^\infty(\Sigma\cap B_\varepsilon(x))$ and $\nu$ is a choice of unit normal, as in \cref{sec:twoside}.
Consider the first nonnegative eigenfunction $\varphi\in C^\infty(\Sigma\cap B_\varepsilon(x))$ relative to the eigenvalue $\mu_1(\Sigma\cap B_\varepsilon(x))$, and define $\tilde\varphi\in C^\infty(\Sigma\cap U_\varepsilon(x))$ as
\[
\tilde\varphi(\selem(x)) \eqdef \operatorname{sgn}_\Sigma(\selem) \varphi(x) 
\]
for all $\selem\in G$. Note that $\tilde\varphi$ is well-defined and $G$-equivariant, i.e., $\tilde\varphi\in C^\infty_G(\Sigma\cap U_\varepsilon(x))$.
As a result, since $\tilde\varphi\in C^\infty_G(\Sigma\cap U_\varepsilon(x))$ is nonnegative and $G$-equivariant, we obtain that $\tilde\varphi$ is the first eigenfunction relative to the eigenvalue $\mu_1(\Sigma\cap U_\varepsilon(x))$, and also relative to $\lambda_1(\Sigma\cap U_\varepsilon(x))$.
This implies that the three numbers $\mu_1(\Sigma\cap B_\varepsilon(x))$, $\mu_1(\Sigma\cap U_\varepsilon(x))$ and $\lambda_1(\Sigma\cap U_\varepsilon(x))$ actually coincide.
\end{proof}

\begin{theorem} \label{thm:ConvBoundGIndex}
In the setting above, let $\{\Sigma_k\}_{k\in\N}$ be a sequence of $G$-equivariant free boundary minimal surfaces in $M$, with uniformly bounded area, such that $\ind_G(\Sigma_k)\le n$ for some fixed $n\in\N$. Then $\Sigma_k$ converges locally graphically and smoothly, possibly with multiplicity, to a free boundary minimal surface $\tilde\Sigma\subset M\setminus(\mathcal{S}\cup \mathcal{Y})$ in $M\setminus(\mathcal{S}\cup \mathcal{Y})$, where $\mathcal{S}\subset M$ is the singular locus of $G$ and $\mathcal{Y}$ is a finite subset of $M$ with $\abs{\mathcal{Y}}\le n\abs{G}$.
Furthermore, if there exists a $G$-equivariant free boundary minimal surface $\Sigma\subset M$ such that $\tilde\Sigma = \Sigma\setminus (\mathcal{S}\cup\mathcal{Y})$ (namely if $\tilde \Sigma$ extends smoothly to $M$), then $\ind_G(\Sigma)\le n$.
\end{theorem}
\vspace{-1.2ex}
\begin{remark} \label{rem:AlsoForConvMetrics}
The previous theorem holds even when $\Sigma_k$ is a free boundary minimal surface with respect to a $G$-equivariant Riemannian metric $\gamma_k$ on $M$, for all $k\in\N$, and the sequence $\{\gamma_k\}_{k\in\N}$ converges smoothly to $\gamma$.
\end{remark}
\begin{proof}
Let us consider the set
\[
\mathcal{Y}\eqdef  \left\{x\in M \setminus\mathcal{S}\st \forall \varepsilon>0,\ \limsup_{k\to\infty} \mu_1(\Sigma_k\cap B_\varepsilon(x)) < 0\right\}\subset M\setminus \mathcal{S}.
\]
By Theorems 18 and 19 in \cite{AmbCarSha18Compactness}, having uniformly bounded area, the surfaces $\Sigma_k$ converge locally graphically and smoothly in $M\setminus(\mathcal{S}\cup\mathcal{Y})$ with finite multiplicity to a free boundary minimal surface $\tilde\Sigma\subset M\setminus(\mathcal{S}\cup\mathcal{Y})$.

Assume by contradictions that $\mathcal{Y}$ contains $n'\eqdef n\abs{G}+1$ distinct points $x_1,\ldots,x_{n'}\in M\setminus\mathcal{S}$. Then, there exists $\varepsilon>0$ sufficiently small, such that the balls $B_\varepsilon(x_1),\ldots,B_\varepsilon(x_{n'})\subset M\setminus\mathcal{S}$ are disjoint and (up to subsequence) $\mu_1(\Sigma_k\cap B_\varepsilon(x_i))<0$ for all $k$ sufficiently large and $i=1,\ldots,n'$.
Now, defining the $G$-equivariant subsets $U_i =U_\varepsilon(x_i)\eqdef \bigcup_{\selem\in G} \selem(B_\varepsilon(x_i))$ for all $i=1,\ldots,n'$ as in \cref{lem:FirstEigenvalues} (taking $\varepsilon$ possibly smaller), we have that $\lambda_1(\Sigma_k\cap U_i) = \mu_1(\Sigma_k\cap B_\varepsilon(x_i))<0$.
Moreover, $U_i\cap U_j\not= \emptyset$ for some $i,j\in\{1,\ldots,n'\}$ if and only if $U_i=U_j$. Note that, fixed $i\in\{1,\ldots,n'\}$, there are at most $\abs{G}$ values of $j\in\{1,\ldots,n'\}$ such that $U_i=U_\varepsilon(x_i) = U_\varepsilon(x_j) = U_j$. Therefore, at least $n+1 = \lceil \frac{n'}{\abs{G}}\rceil$, say $U_1,\ldots,U_{n+1}$, of the sets $U_1,\ldots, U_{n'}$ are disjoint.
In particular, this contradicts the assumption $\ind_G(\Sigma_k) \le n$, since $\Sigma_k\cap U_1,\ldots,\Sigma_k\cap U_{n+1}$ are disjoint $G$-equivariant $G$-unstable subsets of $\Sigma_k$.  
Hence, we have that $\abs{\mathcal{Y}}\le n\abs{G}$.

Let us now suppose that $\tilde\Sigma$ extends smoothly to $\Sigma$ in $M$, and assume by contradiction that $\ind_G(\Sigma)>n$. Then, there exist $G$-equivariant vector fields $X_1,\ldots,X_{n+1}\in\Gamma_G(N\Sigma)$ on $\Sigma$ such that $\sum_{i=1}^{n+1}a_i X_i$ is a negative direction for the second variation of the area of $\Sigma$ for all $(a_1,\ldots,a_{n+1})\in \R^{n+1}\setminus\{0\}$. 
Since the isometries in $G$ are orientation-preserving, $\Sigma\cap\mathcal{S}$ consists of a finite union $\mathcal{S}_0$ of isolated points and of a finite union $\mathcal{S}_1=(\Sigma\cap\mathcal{S})\setminus\mathcal{S}_0$ of smooth curve segments (see \cite{CooperHodgsonKerckhoff2000}*{Theorem 2.4} and \cite{Ketover2016equiv}*{Lemmas 3.3, 3.4 and 3.5}). 
In particular, by Lemmas 3.4 and 3.5 in \cite{Ketover2016equiv}, for all $x\in\mathcal{S}_1$ there exists $\selem\in G$ such that $h(x)=x$ and $h_*\nu = -\nu$, where $\nu$ is a choice of unit normal to $\Sigma$ at the point $x$.
Therefore, by equivariance of $X_1,\ldots,X_{n+1}\in\Gamma_G(N\Sigma)$, we have that $X_i = 0$ on $\mathcal{S}_1$ for all $i=1,\ldots,n+1$.
As a result, thanks to a standard cutoff argument, we can assume without loss of generality that $X_1,\ldots,X_{n+1}$ are compactly supported in $M\setminus(\mathcal{S}\cup \mathcal{Y})$. Indeed, we can first suppose that $X_1,\ldots,X_{n+1}$ are compactly supported in $M\setminus\mathcal{S}_1$, because every continuous function on a surface that is zero along a smooth curve can be approximated in the $H^1$-norm with functions that are zero in a neighborhood of such curve (see e.g. the proof of Theorem 2 in \cite{EvansPDE}*{Section 5.5}).
Secondly, we can suppose that $X_1,\ldots,X_{n+1}$ are zero in a neighborhood of the finite set $\mathcal{S}_0\cup\mathcal{Y}$ using a standard log-cutoff argument. Observe that all these operations can be made in an equivariant way, because both $\mathcal{S}_1$ and $\mathcal{S}_0\cup\mathcal{Y}$ are $G$-equivariant.
Since $\Sigma_k$ smoothly converges (possibly with multiplicity) to $\Sigma$ in $M\setminus (\mathcal{S}\cup \mathcal{Y})$, we thus get that $\ind_G(\Sigma_k) > n$ for $k$ sufficiently large, which contradicts our assumption. This proves $\ind_G(\Sigma)\le n$, as desired.
\end{proof}

\section{Outline of the proof of regularity and genus bound} \label{sec:ProofRegularityGenus}

In this section, we outline the results about the regularity and the genus bound for the surface obtained via a Simon--Smith equivariant min-max procedure. 
Besides being interesting on its own, we also need to unfold some of those arguments for the proof of \cref{thm:EquivMinMax}. Indeed, we do not prove that the surface obtained from the min-max procedure has controlled equivariant index, but we show that it is possible to properly modify the proof in order to obtain a surface with the desired bound on the equivariant index.

We also include some detailed proofs for completeness, even if they are very similar to existing ones. This is because sometimes we need slight variations of existing results (for example here we consider $n$-dimensional sweepouts) or the proofs are split in many different references.

\begin{definition}
In the setting of \cref{thm:EquivMinMax}, given a minimizing sequence of sweepouts $\so^j=\{\Sigma^j_t\}_{t\in I^n}\in$ for $j\in\N$ (namely such that $\lim_{j\to\infty}\sup_{t\in I^n} \area(\Sigma_t^j) = W_\Pi$), we define its \emph{critical set} as
\begin{align*}
    C(\{\so^j\}_{j\in\N}) = \{ & V\in\mathcal{V}^2_G(M) \text{ such that $\exists$ $j_k\to+\infty$, $t_{j_k}\in I^n$ with $\vard(\Sigma^{j_k}_{t_{j_k}},V)\to 0$},\ \norm{V}(M)=W_\Pi\}.
\end{align*}
\end{definition}

\subsection{Min-max sequence converging to a stationary varifold}

The first step is to show that there exists a min-max sequence converging to a stationary varifold, as stated in the following proposition.

\begin{proposition} [{cf. \cite{Ketover2016equiv}*{Proposition 3.9}}] \label{prop:ConvergenceToStationary}
In the setting of \cref{thm:EquivMinMax}, given a minimizing sequence $\{\boldsymbol{\Lambda}^j\}_{j\in\N}=\{\{\Lambda_t^j\}_{t\in I^n}\}_{j\in\N}$, there exists another minimizing sequence $\{\so^j\}_{j\in\N}=\{\{\Sigma_t^j\}_{t\in I^n}\}_{j\in \N}$ such that, if $\{\Sigma_{t_j}^j\}_{j\in\N}$ is a min-max sequence, then (up to subsequence) $\Sigma^j_{t_j}$ converges in the sense of varifolds to a stationary varifold.
Moreover, it holds $C(\{\so^j\}_{j\in\N}) \subset C(\{\boldsymbol{\Lambda}^j\}_{j\in\N})$.
\end{proposition}
\begin{proof}
We include the proof for completeness. However, the argument is the same as in the proof of \cite{ColdingDeLellis2003-Errata}*{Proposition 2.1}, given the suitable modifications to the equivariant setting.

We first restrict to the compact space
\[
\mathcal{X} \eqdef \{V\in \mathcal{V}^2_G(M) \st \norm{V}(M)\le 2W_\Pi\}
\]
and we denote by $\mathcal{V}_\infty$ the subspace of $G$-stationary varifolds of $\mathcal{X}$, which coincides with the subspace of stationary varifolds by \cref{rem:stationaryGstationary}, and we define $\mathcal{Y}\eqdef \mathcal{X}\setminus \mathcal{V}_\infty$.
Now, for every $V\in \mathcal{Y}$, consider a $G$-equivariant smooth vector field $\xi_V\in\vf_G(M)$ such that $\delta V(\xi_V) < 0$. Up to rescaling, we can assume without loss of generality that $\norm{\xi_V}_{C^k}\le 1/k$, whenever $\vard(V,\mathcal{V}_\infty)\le 2^{-k}$ and $k\ge 1$ is any positive integer.

Then, for every $V\in \mathcal{Y}$, let $0 < \rho(V) < \vard(V,\mathcal{V}_\infty)/2$ be such that $\delta W(\xi_V)<0$ for every $W\in B_{\rho(V)}^{\vard}(V)$. Note that $\{B_{\rho(V)}^{\vard}(V)\}_{V\in \mathcal{Y}}$ is a cover of $\mathcal{Y}$ and thus it admits a subordinate partition of unity $\{\varphi_V\}_{V\in \mathcal{Y}}$. Therefore, for every $V\in \mathcal{Y}$, we can define the smooth $G$-equivariant vector field
\[
H_V\eqdef \sum_{W\in \mathcal{Y}} \varphi_W(V) \xi_W.
\]
Note that the map $\mathcal{Y}\ni V \mapsto H_V \in \vf_G(M)$ is continuous. Moreover, observe that $\delta V(H_V) < 0$ for every $V\in \mathcal{Y}$ and that, if $\vard(V,\mathcal{V}_\infty) \le 2^{-k-1}$ for some positive integer $k\ge 1$, then $\norm{H_V}_{C^k}\le 1/k$. 
This last property follows from the fact that, if $\vard(V,\mathcal{V}_\infty) \le 2^{-k-1}$ for some positive integer $k\ge 1$ and $\varphi_W(V) > 0$ for some $W\in \mathcal{Y}$, then $\vard(W,\mathcal{V}_\infty)\le 2\, \vard(V,\mathcal{V}_\infty) \le 2^{-k}$ and therefore $\norm{\xi_W}_{C^k}\le 1/k$.

As a result, the map $V\mapsto H_V$ defined on $\mathcal{Y}$ can be extended to a continuous function $\mathcal{X}\to \vf_G(M)$ by setting it identically equal to $0$ on $\mathcal{V}_\infty$.
Then, for every $V\in \mathcal{X}$, the flow $\Psi_V\colon[0,+\infty)\to \operatorname{Diff}(M)$ generated by $H_V$ is a $G$-equivariant isotopy. Since the map $V\in \mathcal{X}\mapsto H_V\in \vf_G(M)$ is continuous, then also the map $V\mapsto \Psi_V$ is continuous. In particular the function $V\mapsto \delta(\Psi_V(s,\cdot)_\sharp V)(H_V)$ is continuous as well for all $s\ge 0$ and $\Psi_V(0,\cdot)_\sharp V = V$. 
Hence, for every $V\in \mathcal{Y}$, let $0<\tau(V)\le 1$ be the maximum such that
\[
\delta(\Psi_V(s,\cdot)_\sharp V)(H_V) \le \frac 12 \delta V(H_V) < 0 \ \text{for all $s\in [0,\tau(V)]$}.
\]
The map $\tau\colon \mathcal{Y}\to [0,1]$ is continuous, hence we can redefine $H_V$ multiplying it by $\tau(V)$ in $\mathcal{Y}$, i.e., we reset it to be equal to $\tau(V)H_V$ in $\mathcal{Y}$ and equal to $0$ elsewhere. Note that this gives a well-defined vector field $\mathcal{X}\ni V\mapsto H_V\in \vf_G(M)$, because $\tau$ is bounded in $\mathcal{Y}$ and therefore the newly-defined $H_V$ is continuous if we set it equal to $0$ in $\mathcal{V}_\infty$.
As a result, we also redefine the flow $\Psi_V$ and it holds
\[
\delta(\Psi_V(s,\cdot)_\sharp V) (H_V) < 0 \ \text{for all $V\in \mathcal{Y}$ and $s\in [0,1]$}.
\]

Now let $\{\boldsymbol{\Lambda}^j\}_{j\in\N}=\{\{\Lambda_t^j\}_{t\in I^n}\}_{j\in\N}\subset \Pi$ be a minimizing sequence and define the $G$-equivariant surface $\tilde\Sigma_t^j \eqdef \Psi_{\Lambda_t^j}(1,\Lambda_t^j)$ for all $j\in \N$ and $t\in I^n$. Observe that the map $(t,x)\mapsto \tilde \Phi^j(t,x)\eqdef \Psi_{\Lambda_t^j}(1,x)$ is $C^1$ in the parameter $t$;
however, it is not necessarily smooth, hence $\{\tilde\Sigma_t^j\}_{t\in I^n}$ is not necessarily contained in $\Pi$.
Therefore, for all $j\in\N$, let us define a smooth map $(t,x)\mapsto \Phi^j(t,x)$ by convolving $\tilde \Phi^j$ with a smooth kernel in the parameter $t\in I^n$ in such a way that
\begin{equation} \label{eq:hnearc1}
\norm{\Phi^j(t,\cdot)-\tilde \Phi^j(t,\cdot)}_{C^1} \le \frac{1}{j+1}.
\end{equation}
Note that $\Phi^j(t,\cdot)$ is $G$-equivariant, because we are convoluting in the parameter $t$ and $\tilde\Phi^j(t,\cdot)$ is $G$-equivariant.
Then, define the $G$-equivariant surface $\Sigma_t^j\eqdef \Phi^j(t,\Lambda_t^j)$ for all $j\in\N$ and $t\in I^n$. By smoothness of the map $t\mapsto \Phi^j(t,\cdot)$, we have that $\{\Sigma_t^j\}_{t\in I^n}\in\Pi$ for all $j\in\N$.
Moreover, by \eqref{eq:hnearc1}, it holds that
\begin{equation} \label{eq:TildeNearSigma}
\lim_{j\to+\infty} \sup_{t\in I^n} \vard(\Sigma_t^j,\tilde \Sigma_t^j) = 0.
\end{equation}
As a result, using that $\area(\tilde\Sigma_t^j) = \area(\Psi_{\Lambda_t^j}(1,\Lambda_t^j)) \le \area(\Lambda_t^j)$, we have that
\[
m_0\le \limsup_{j\in\N}\sup_{t\in I^n} \area(\Sigma_t^j) = \limsup_{j\in\N}\sup_{t\in I^n} \area(\tilde\Sigma_t^j) \le \limsup_{j\in\N}\sup_{t\in I^n} \area(\Lambda_t^j) = m_0,
\]
thus, a posteriori, all the inequalities are equalities; in particular $\limsup_{j\in\N}\sup_{t\in I^n} \area(\Sigma_t^j) = m_0$.

Our aim now is to prove that the minimizing sequence $\{\boldsymbol{\Sigma}^j\}_{j\in\N}=\{\{\Sigma_t^j\}_{t\in I^n}\}_{j\in \N}$ has the properties claimed in the statement. Let $t_j\in I^n$ be such that $\{\Sigma_{t_j}^j\}_{j\in\N}$ is a min-max sequence. Up to subsequence, $\Sigma_{t_j}^j$ converges in the sense of varifolds to a varifold $V$, which coincides with the limit in the sense of varifolds of $\tilde\Sigma_{t_j}^j$ by \eqref{eq:TildeNearSigma}.
Possibly passing to a subsequence, we can also assume that $\Lambda_{t_j}^j$ converges in the sense of varifolds to some varifold $W$ and, by continuity of the map $\Psi$, we have that $\Psi_W(1,\cdot)_\sharp W = V$.

First consider the case when $W$ is not stationary.
Then $\norm{W}(M) = m_0$, but on the other hand it holds
\begin{align*}
m_0 &= \lim_{j\to+\infty} \area(\tilde\Sigma_{t_j}^j) = \norm{V}(M) = \norm{\Psi_W(1,\cdot)_\sharp W}(M) \\
&= \norm{W}(M) + \int_0^1 \delta(\Psi_W(s,\cdot)_\sharp W)(H_W) \de s < m_0,
\end{align*}
which is a contradiction. As a result $W$ is stationary and therefore $W=\Psi_W(1,\cdot)_\sharp W= V$, which implies that $V$ is stationary, as desired. Moreover, since $\Lambda_{t_j}^j\to W =V$, we also get that $C(\{\so^j\}_{j\in\N}) \subset C(\{\boldsymbol{\Lambda}^j\}_{j\in\N})$.
\end{proof}

\subsection{Min-max sequence $G$-almost minimizing in annuli}

The second step is to prove that we can further assume that the min-max sequence found in the previous section has a certain ``regularizing property'', namely that of being $G$-almost minimizing in suitable subsets. The precise statements and definitions are given below.

\begin{definition}
Given $\varepsilon >0$, a $G$-equivariant open set $U\subset M$ (namely $\selem(U) = U$ for all $\selem\in G$), and a $G$-equivariant surface $\Sigma\subset U$, we say that $\Sigma$ is \emph{$(G,\varepsilon)$-almost minimizing} in $U$ if there does not exist any $G$-equivariant isotopy $\Psi\colon[0,1]\to\operatorname{Diff}(M)$ supported in $U$ such that
\begin{itemize}
\item $\area(\Psi(s,\Sigma)) \le \area(\Sigma) + \varepsilon/2^{n+2}$ for all $s\in[0,1]$;
\item $\area(\Psi(1,\Sigma)) \le \area(\Sigma) - \varepsilon$.
\end{itemize}
A sequence $\{\Sigma^j\}_{j\in\N}$ is said to be \emph{$G$-almost minimizing} in $U$ if there exists a sequence of positive numbers $\varepsilon_j\to 0$ such that $\Sigma^j$ is $(G,\varepsilon_j)$-almost minimizing in $U$.
\end{definition}
\begin{remark}
Observe that the constants in the definition above differ from the ones in \cite{ColdingDeLellis2003}*{Definition 3.2}. This is due to the fact that we are considering $n$-dimensional sweepouts and thus we need these different constants for the proof of \cref{lem:AlmostMinSeqInAdmissible}.
\end{remark}

\begin{definition}
We say that a subset $\mathrm{An}\subset M$ is a \emph{$G$-equivariant annulus} if there are $x\in M$ and $0<r<s$ such that $\mathrm{An} = \bigcup_{g\in G} g(\mathrm{An}(x,r,s))$, where $\mathrm{An}(x,r,s) = B_s(x) \setminus \overline{B_r(x)}$ is the open annulus of center $x$, inner radius $r$ and outer radius $s$. 
Note that, for $r,s$ sufficiently small $\bigcup_{g\in G} g(\mathrm{An}(x,r,s))$ consists of a finite union of disjoint open annuli (note that some annuli in the union could possibly overlap). 
Given $x\in M$, we denote by $\mathcal{AN}^G_{r}(x)$ the union of all $G$-equivariant annuli $\bigcup_{g\in G} g(\mathrm{An}(x,r_1,r_2))$ as the radii $r_1,r_2$ vary in the range $0<r_1<r_2\le r$.

We say that a family of $G$-equivariant annuli is \emph{$L$-admissible} for some positive integer $L$, if it consists of $L$ $G$-equivariant annuli $\mathrm{An}_1,\ldots, \mathrm{An}_L$ such that $\mathrm{An}_i = \bigcup_{g\in G} g(\mathrm{An}(x,r_i,s_i))$ for some $x\in M$ and $r_i,s_i>0$ with $r_{i+1}>2s_i$ for all $i=1,\ldots,L-1$.
\end{definition}
\begin{definition}
We say that a surface $\Sigma\subset M$ is $(G,\varepsilon)$-almost minimizing in an $L$-admissible family $\mathcal{A}$ of $G$-equivariant annuli if it is $(G,\varepsilon)$-almost minimizing in at least one $G$-equivariant annulus of $\mathcal{A}$. Moreover, we say that a sequence $\{\Sigma^j\}_{j\in\N}$ is $G$-almost minimizing in an $L$-admissible family $\mathcal{A}$ of $G$-equivariant annuli if there exists a sequence $\varepsilon_j\to 0 $ such that $\Sigma^j$ is $(G,\varepsilon_j)$-almost minimizing in $\mathcal{A}$; equivalently, $\{\Sigma^j\}_{j\in\N}$ is $G$-almost minimizing in $\mathcal{A}$ if it is $G$-almost minimizing in at least one $G$-equivariant annulus of $\mathcal{A}$.
\end{definition}

\begin{lemma} [{cf. \cite{ColdingGabaiKetover2018}*{Lemma A.1}}] \label{lem:AlmostMinSeqInAdmissible}
Let $\{\so^j\}_{j\in\N}=\{\{\Sigma_t^j\}_{t\in I^n}\}_{j\in \N}$ be a minimizing sequence obtained from \cref{prop:ConvergenceToStationary},
then there exists $t_j\in I^n$ for all $j\in\N$ such that $\{\Sigma^j_{t_j}\}_{j\in\N}$ is a $G$-equivariant min-max sequence that is $G$-almost minimizing in every $L$-admissible family of $G$-equivariant annuli, where $L=(3^n)^{3^n}$.
\end{lemma}
\begin{proof}
Let $\{\so^j\}_{j\in\N}$ be the minimizing sequence given by \cref{prop:ConvergenceToStationary}. Fixed any $J>0$, we prove that there exist $j>J$ and $t_j\in I^n$ such that $\Sigma^j_{t_j}$ is $(G,1/J)$-almost minimizing in every $L$-admissible family of $G$-equivariant annuli and $\area(\Sigma^j_{t_j}) \ge m_0 - 1/J$.
To this purpose, consider any $j>J$ and define $K\eqdef \{t\in I^n\st \area(\Sigma^j_t) \ge m_0 - 1/J\}$. Assume that for all $t\in K$ there exists an $L$-admissible family $\mathcal{A}_t=\{\mathrm{An}_t^1,\ldots,\mathrm{An}_t^L\}$ of $G$-equivariant annuli where $\Sigma_t^j$ is not $(G,1/J)$-almost minimizing, namely for all $t\in K$ and $i=1,\ldots,L$ there exists a $G$-equivariant isotopy $\Psi_t^i\colon[0,1]\to\operatorname{Diff}(M)$ supported in $\mathrm{An}_t^i$ such that
\begin{itemize}
    \item $\area(\Psi_t^i(s,\Sigma_t^j)) \le \area(\Sigma_t^j) + 1/(2^{n+2}J)$ for all $s\in[0,1]$;
    \item $\area(\Psi_t^i(1,\Sigma_t^j)) \le \area(\Sigma_t^j) - 1/J$.
\end{itemize}
By continuity in the sense of varifolds of the map $t\mapsto \Sigma_t^j$, for every $t\in K$ we can find a neighborhood $U_t\subset I^n$ such that
\begin{itemize}
    \item $\area(\Psi_t^i(s,\Sigma_r^j)) \le \area(\Sigma_r^j) + 1/(2^{n+1}J)$ for all $s\in[0,1]$, $r\in U_t$;
    \item $\area(\Psi_t^i(1,\Sigma_r^j)) \le \area(\Sigma_r^j) - 1/(2J)$ for all $r\in U_t$.
\end{itemize}
Since $\{U_t\}_{t\in K}$ is an open cover of $K$ and $K$ is compact, we can extract a finite subcover $\mathcal{U}$. We now want to find another cover $\mathcal{O}$ of $K$ and for every $O\in\mathcal{O}$ we want to choose a $G$-equivariant annulus $\mathrm{An}_O$ such that
\begin{enumerate}[label={\normalfont(\roman*)}]
\item\label{opencoverK:i} for all $O\in\mathcal{O}$ there exists $U_t\in\mathcal{U}$ such that $O\subset U_t$ and $\mathrm{An}_O = \mathrm{An}_t^i$ for some $i=1,\ldots,L$;
\item\label{opencoverK:ii} for all $O_1,O_2\in\mathcal{O}$ with nonempty intersection we have that $\mathrm{An}_{O_1} \cap \mathrm{An}_{O_2} = \emptyset$.
\end{enumerate}

To this purpose, we follow \cite{Pitts1981}*{Chapter 4}). For all $m\in\N$, let $I(1,m)$ be the CW complex structure, supported on the $1$-dimensional interval $I^1=[0,1]$, whose $0$-cells are $[0], [3^{-m}], [2\cdot 3^{-m}], \ldots, [1-3^{-m}],[1]$ and whose $1$-cells are the intervals $[0,3^{-m}], [3^{-m}, 2\cdot 3^{-m}], \ldots, [1-3^{-m}, 1]$. Then consider the CW complex structure $I(n,m)$ on the $n$-dimensional interval $I^n$, whose $p$-cells for $0\le p\le n$ can be written as $\sigma_1\otimes\sigma_2\otimes\ldots\otimes \sigma_n$, where $\sigma_1,\ldots,\sigma_n$ are cells of the CW complex $I(1,m)$ such that $\sum_{i=1}^n \dim(\sigma_i) = p$.

For every cell $\sigma = \sigma_1\otimes\sigma_2\otimes\ldots\otimes \sigma_n\in I(n,m)$, let $T(\sigma)$ be the open subset of $I_n$ defined as $T(\sigma)\eqdef T(\sigma_1)\times T(\sigma_2)\times \ldots \times T(\sigma_n)$, where 
\[
T(\sigma_i) \eqdef \begin{cases}
                       (x-3^{-m-1},x+3^{-m-1}) & \text{if $\sigma_i = [x]$ is a $0$-cell}\\
                       (x,y) & \text{if $\sigma_i = [x,y]$ is a $1$-cell}. 
                   \end{cases}
\]
Note that $\{T(\sigma)\st \sigma\in I(n,m), T(\sigma)\cap K \not=\emptyset\}$ is an open cover of $K$. Then define $\mathcal{O}$ to be such an open cover for some $m$ sufficiently large such that $3^{-m}\sqrt{n}$ is less than the Lebesgue number of $\mathcal{U}$, thus $O$ has diameter less than the Lebesgue number of $\mathcal{U}$ for all $O\in\mathcal{O}$.
Then, for all $O\in\mathcal{O}$, let $\mathcal{A}_O$ be an $L$-admissible family of $G$-equivariant annuli in $M$ such that there exists $U_t\in\mathcal{U}$ with $O\subset U_t$ and $\mathcal{A}_O=\mathcal{A}_t$.

Now, note that $T(\sigma_1)\cap T(\sigma_2)\not=\emptyset$ for some $\sigma_1,\sigma_2\in I(n,m)$ if and only if $\sigma_1$ and $\sigma_2$ are faces (possibly of different dimension) of a cell in $I(n,m)$. 
We can thus exploit the following proposition to associate to every $O=T(\sigma)\in\mathcal{O}$ a $G$-equivariant annulus $\mathrm{An}_O\in\mathcal{A}_O$.

\begin{proposition} [\cite{Pitts1981}*{Proposition 4.9}] \label{prop:PittsChoice}
For every $\sigma\in I(n,m)$, let $\mathcal{A}(\sigma)$ be an $L$-admissible family of $G$-equivariant annuli in $M$ with $L=(3^n)^{3^n}$. Then for every $\sigma\in I(n,m)$ it is possible to choose a $G$-equivariant annulus $\mathrm{An}(\sigma)\in\mathcal{A}(\sigma)$ such that $\mathrm{An}(\sigma_1)\cap \mathrm{An}(\sigma_2)$ for every $\sigma_1\not=\sigma_2\in I(n,m)$ that are faces of a cell in $I(n,m)$.
\end{proposition}

Hence, for all $O = T(\sigma)\in \mathcal{O}$, let us define $\mathrm{An}_O \eqdef \mathrm{An}(\sigma)$, where $\mathrm{An}(\sigma)$ is given from \cref{prop:PittsChoice}. This proves the existence of the open cover $\mathcal{O}$ of $K$ with properties \ref{opencoverK:i} and \ref{opencoverK:ii} above. Moreover, by definition of $\mathcal{U}$ and $\mathcal{O}$, for all $O\in\mathcal{O}$ there exists a $G$-equivariant isotopy $\Psi_O\colon[0,1]\to \operatorname{Diff}(M)$ supported in $\mathrm{An}_O$ such that
\begin{itemize}
\item $\area(\Psi_O(s,\Sigma_r^j))\le \area(\Sigma_r^j) + 1/(2^{n+1}J)$ for all $s\in[0,1]$, $r\in O$;
\item $\area(\Psi_O(1,\Sigma_r^j))\le \area(\Sigma_r^j) - 1/(2J)$ for all $r\in O$.
\end{itemize}

Now, for all $O\in \mathcal O$ choose a $C^\infty$ function $\varphi_O\colon I^n\to [0,1]$ supported in $O$, such that for all $s\in K$ there exists $O\in\mathcal{O}$ such that $\varphi_O(s) = 1$.
Then, given any $t\in I^n$, consider the function $\Phi_t\in \operatorname{Diff}(M)$ defined as
\[
\Phi_t(x) \eqdef \begin{cases}
                     \Psi_O(\varphi_O(t), x) & \text{if $x\in \mathrm{An}_O$ for some $O\in\mathcal{O}$ with $t\in O$,}\\
                     x &\text{if $t\not\in O$ for all $O\in\mathcal O$.}
                 \end{cases}
\]
Observe that $\Phi_t$ is well-defined since $\mathrm{An}_{O_1}\cap \mathrm{An}_{O_2}=\emptyset$ whenever $O_1,O_2\in\mathcal O$ are not disjoint.
Moreover, $t\mapsto \Phi_t$ is a $G$-equivariant isotopy with $\Phi_t=\id$ for $t\in\partial I^n$. Indeed, note that $d(K,\partial I^n) > 0$, hence we can choose $m$ possibly larger in such a way that $O\cap \partial I^n=\emptyset$ for all $O\in\mathcal O$.
Therefore, we can define $\tilde\Sigma_t^j = \Phi_t(\Sigma_t^j)$ and we have that $\{\tilde\Sigma_t^j\}_{t\in I^n}\in \Pi$. 
Now note that, for every $t\in K$, there exists $O\in \mathcal O$ such that $\varphi_O(t) = 1$, moreover $t\in O$ for at most other $2^n-1$ sets $O\in\mathcal O$. As a result, we can estimate
\[
\area(\tilde\Sigma_t^j) \le \area(\Sigma_t^j) - \frac{1}{2J} + \frac{2^n-1}{2^{n+1}J} =\area(\Sigma_t^j) - \frac{1}{2^{n+1}J}
\]
for all $t\in K$, and
\[
\area(\tilde\Sigma_t^j) \le \area(\Sigma_t^j) + \frac{2^n}{2^{n+1}J} < m_0-\frac 1{2J}
\]
for all $t\in I^n\setminus K$.
By arbitrariness of $j$, this would give $\liminf_{j\to+\infty}\sup_{t\in I^n} \area(\tilde\Sigma_t^j) < m_0$, which is a contradiction. This proves that, for every $J>0$, there exists $j>J$ and $t_j\in I^n$ such that $\Sigma_{t_j}^j$ is $(G,1/J)$-almost minimizing in every $L$-admissible family of $G$-equivariant annuli and $\area(\Sigma_{t_j}^j) \ge m_0-1/J$. In particular, we have that $\{\Sigma^j_{t_j}\}_{j\in\N}$ is a $G$-equivariant min-max sequence that is $G$-almost minimizing in every $L$-admissible family of $G$-equivariant annuli.
\end{proof}

\subsection{Regularity and genus bound as a consequence of $G$-almost minimality}\label{sec:ProofRegularityGenusFinal}

The third and final step is to show that the properties of the min-max sequence obtained in the previous sections are sufficient to prove that this sequence converges to a $G$-equivariant free boundary minimal surface (possibly with multiplicity) for which the genus bounded in \cref{thm:EquivMinMax} holds.

\begin{proposition} \label{prop:AlmostMinSeq}
In the setting of \cref{thm:EquivMinMax}, given a $G$-equivariant min-max sequence $\{\Sigma^j\}_{j\in\N}$ that is $G$-almost minimizing in every $L$-admissible family of $G$-equivariant annuli for $L=(3^n)^{3^n}$ (as the one obtained in \cref{lem:AlmostMinSeqInAdmissible}), there exists a $G$-equivariant function $r\colon M \to \R^+$ such that (up to subsequence)  $\{\Sigma^j\}_{j\in\N}$ is $G$-almost minimizing in $\mathrm{An}\in \mathcal{AN}_{r(x)}^G(x)$ for all $x\in M$.
Moreover, $\{\Sigma^j\}_{j\in\N}$ converges in the sense of varifolds as $j\to+\infty$ to $\limitv = \bigcup_{i=1}^k m_i\limitv_i$, where $\limitv_i$ is a $G$-equivariant free boundary minimal surface and $m_i$ is a positive integer for all $i=1,\ldots,k$, and the genus bound in \cref{thm:EquivMinMax} holds.
\end{proposition}
\begin{proof}
Thanks to \cite{ColdingGabaiKetover2018}*{Lemma A.3}, we have that there exists a $G$-equivariant function $r\colon M\to \R^+$ such that (up to subsequence) $\{\Sigma^j\}_{j\in\N}$ is $G$-almost minimizing in every $\mathrm{An}\in \mathcal{AN}_{r(x)}^G(x)$ for all $x\in M$. Indeed, the only difference in the proof is that, instead of considering annuli $\mathrm{An}(x,r,s)$ centered at some point $x\in M$, here we have to consider $G$-equivariant annuli $\bigcup_{g\in G}g(\mathrm{An}(x,r,s))$ ``centered'' at some point $x\in M$.

At this point the proof of the second part of the statement follows using the $G$-almost minimality of $\{\Sigma^j\}_{j\in\N}$.
Indeed, given a $G$-equivariant min-max sequence $\{\Sigma^j\}_{j\in\N}$ that is $G$-almost minimizing in $\mathrm{An}\in \mathcal{AN}_{r(x)}^G(x)$ for all $x\in M$, we get that (up to subsequence) $\Sigma^j$ converges as $j\to+\infty$ to $\limitv$, which is a finite union of $G$-equivariant free boundary minimal surfaces (possibly with multiplicity), and the genus bound in \cref{thm:EquivMinMax} holds.
The proof of the regularity of $\limitv$ can be found in \cite{ColdingDeLellis2003}*{Theorem 7.1} for the closed case and in \cite{Li2015}*{Proposition 4.11} for the free boundary case.
The proof of the genus bound is instead contained in \cite{DeLellisPellandini2010}, \cite{Ketover2019}. 
The adaptations to the equivariant setting are the object of \cite{Ketover2016equiv}*{Section 4, after Proposition 4.12} and \cite{Ketover2016fb}*{Section 7.2}.
\end{proof}

\begin{remark} \label{rem:SymmetriesLocally}
Observe that the difficulties in the equivariant setting arise around the points in the singular locus of $G$, i.e., points $x\in M$ such that there exists $h\in G$ with $h(x)=x$. 
Since we required the isometries in $G$ to be orientation-preserving, we have that the limit surface can intersect the singular locus only in a point where, locally, the symmetry group is conjugate to $\Z_k$ or $\mathbb D_k$ for some $k\ge 2$ (see \cite{CooperHodgsonKerckhoff2000}*{Theorem 2.4} or \cite{Ketover2016equiv}*{Lemma 3.3}). This is the reason why we add the assumption of ``orientation-preserving'' to the isometries in $G$, namely to exclude the case in which locally the singular locus is a plane.
\end{remark}

\section{Deformation theorem}

In this section we prove that, in the setting of \cref{thm:EquivMinMax}, it is possible to modify a minimizing sequence in a way that, so to say, it avoids a given free boundary minimal surface with $G$-equivariant index greater or equal than $n+1$. The deformation theorem, \cref{thm:Deformation}, is the analogue of Deformation Theorem A in \cite{MarquesNeves2016}.
Before presenting it, we recall three lemmas contained in \cite{MarquesNeves2016} that are needed in the proof.
\begin{lemma} \label{lem:IndexDiffeo}
Let $(M^3,g)$ be a three-dimensional Riemannian manifold with strictly mean convex boundary and let $G$ be a finite group of isometries of $M$.
Given a finite union $\avoids^2\subset M$ of $G$-equivariant free boundary minimal surfaces (possibly with multiplicity) with $\ind_G(\operatorname{spt}(\avoids))\ge n+1$, there exist $0<c_0<1$, $\delta>0$ and a smooth family $\{F_v\}_{v\in \overline{B}^{n+1}}\subset \operatorname{Diff}_G(M)$ of $G$-equivariant diffeomorphisms with:
\begin{enumerate}[label={\normalfont(\roman*)}]
\item $F_0=\id$, $F_{-v} = F_v^{-1}$ for all $v\in\overline{B}^{n+1}$;
\item for any $V\in \overline{B}_{2\delta}^{\vard}(\avoids)$, the smooth function $A^V\colon\overline{B}^{n+1}\to[0,+\infty)$ given by \[ A^V(v) = \norm{(F_v)_\# V}(M)\] has a unique maximum at $m(V)\in B^{n+1}_{c_0/\sqrt{10}}(0)$ and it satisfies $-c_0^{-1}\id \le \Diff^2 A^V(v) \le -c_0 \id$ for all $v\in\overline{B}^{n+1}$.
\end{enumerate}
\end{lemma}
\begin{proof}
Using that $\ind_G(\operatorname{spt}(\avoids))\ge n+1$, we can find $G$-equivariant normal vector fields $X_1,\ldots,X_{n+1}$ on $\operatorname{spt}(\avoids)$ such that 
\[
Q^{\operatorname{spt}(\avoids)}\left(\sum_{i=1}^{n+1}a_iX_i,\sum_{i=1}^{n+1}a_iX_i\right) < 0
\]
for all $(a_1,\ldots,a_{n+1})\not=0\in\R^{n+1}$. These vector fields can be extended to $G$-equivariant vector fields defined in all $M$ (see \cref{rem:GequivExtension}).
Furthermore recall that, given the flow $\Phi\colon[0,+\infty)\times M\to M$ of a $G$-equivariant vector field $Y$, we have that $\Phi_t\in \operatorname{Diff}_G(M)$ for all $t\in[0,+\infty)$ (see \cref{rem:GequivVecFieldToFlow}). 
Given these observations, the proof follows exactly as in \cite{MarquesNeves2016}*{Proposition 4.3}, see also \cite{MarquesNeves2016}*{Definition 4.1}.
\end{proof}

The following two lemmas coincide exactly with \cite{MarquesNeves2016}*{Lemmas 4.5 and 4.4}, since the $G$-equivariance does not play any role in these two results. We report them here for the sake of expository convenience.
\begin{lemma}[{\cite{MarquesNeves2016}*{Lemma 4.5}}] \label{lem:PushAway}
In the setting of \cref{lem:IndexDiffeo}, for every $G$-equivariant varifold $V\in \overline{B}_{2\delta}^{\vard}(\avoids)$, let $\Phi^V\colon[0,+\infty)\times \overline{B}^{n+1}\to\overline{B}^{n+1}$ be the one-parameter flow generated by the vector field
\[
u\mapsto -(1-\abs{u}^2)\grad A^V(u), \quad u\in \overline{B}^{n+1},
\]
as defined also in \cite{MarquesNeves2016}*{pp. 476}.
Then, for all $0<\eta<1/4$, there is $T=T(\eta,\delta,\avoids,\{F_v\}_{v\in \overline{B}^{n+1}}, c_0) \ge 0$ such that, for all $V\in \overline{B}_{2\delta}^{\vard}(\avoids)$ and $v\in \overline{B}^{n+1}$ with $\abs{v-m(v)}\ge \eta$, we have 
\[
A^V(\Phi^V(T,v)) < A^V(0) - \frac{c_0}{10} \quad\text{and} \quad \abs{\Phi^V(T,v)} >\frac{c_0}{4}.
\]
\end{lemma}
\begin{remark}
Note that $\Phi^V$ is smooth since $A^V$ is. Moreover, for all $u\in \overline{B}^{n+1}$, the map $s\mapsto A^V(\Phi^V(s,u))$ is nonincreasing.
\end{remark}

\begin{lemma}[{\cite{MarquesNeves2016}*{Lemma 4.4}}] \label{lem:FarIsFar}
There exists $\overline{\eta}=\overline{\eta}(\delta,\avoids,\{F_v\}_{v\in \overline{B}^{n+1}}) >0$ such that, for any $G$-equivariant varifold $V\in (B_\delta^{\vard}(\avoids))^c$ with
\[
\norm{(F_v)_\# V}(M)\le \norm{V}(M) + \overline{\eta}
\]
for some $v\in \overline{B}^{n+1}$, we have $\vard((F_v)_\#V, \avoids)\ge 2\overline{\eta}$.
\end{lemma}

\begin{theorem}[Deformation theorem] \label{thm:Deformation}
Let $\{\so^j\}_{j\in\N}=\{\{\Sigma^j_t\}_{t\in I^n}\}_{j\in\N}$ be a minimizing sequence in the setting of \cref{thm:EquivMinMax}. Moreover, assume that
\begin{enumerate}[label={\normalfont(\roman*)}]
\item $\avoids^2$ is a finite union of $G$-equivariant free boundary minimal surfaces (possibly with multiplicity) with $\ind_G(\operatorname{spt}(\avoids))\ge n+1$;
\item $\area({\avoids}) = W_\Pi$;
\item $K$ is a compact set of varifolds such that $\avoids\not\in K$ and $\Sigma_t^j\not\in K$ for all $j\in\N$, $t\in I^n$.
\end{enumerate}
Then there exist $\varepsilon>0$ and another minimizing sequence $\{\boldsymbol{\Lambda}^j\}_{j\in\N}=\{\{\Lambda_t^j\}_{t\in I^n}\}_{j\in\N}\subset \Pi$ such that $\Lambda_t^j\cap(\overline{B}^{\vard}_\varepsilon(\avoids)\cup K) =\emptyset$ for all $j$ sufficiently large.
\end{theorem}

\begin{proof}
As aforementioned, the result is the analogue of Deformation Theorem A in \cite{MarquesNeves2016} and the idea of the proof is to exploit the fact that we have many negative directions for the second variation of the area along $\avoids$, hence it is possible to push the minimizing sequence $\{\so^j\}_{j\in\N}$ away from $\avoids$, keeping the fact that it is a minimizing sequence. Indeed, the sweepout is $n$-dimensional, while we have $n+1$ negative directions.

Since $\operatorname{spt}(\avoids)$ has $G$-equivariant index greater or equal than $n+1$, we can apply \cref{lem:IndexDiffeo} and obtain $0<c_0<1$, $\delta>0$ and the family $\{F_v\}_{v\in\overline{B}^{n+1}}\subset \operatorname{Diff}_G(M)$ given by the lemma. 
Moreover, up to modifying $\delta$ and $\{F_v\}_{v\in\overline{B}^{n+1}}$, we can assume that 
\begin{equation} \label{eq:AvoidK}
\vard(\avoids,F_v(\tilde\avoids)) \le \vard(\avoids, K)/2 \quad\text{for all $\tilde\avoids\in \overline{B}^{\vard}_{2\delta}(\avoids)$ and $v\in \overline{B}^{n+1}$}.
\end{equation}

Fixed $j\in\N$, define the open subset $U_j\subset I^n$ given by
\[
U_j\eqdef \{t\in I^n\st \vard(\avoids,\Sigma_t^j)<7\delta/4\}.
\]
Consider the continuous function $m_j\colon U_j\to B^{n+1}_{c_0/\sqrt{10}}(0)$ given by $m_j(t) = m(\Sigma_t^j)$, where the function $m$ is defined in \cref{lem:IndexDiffeo}. Since $\dim U_j = n < n+1= \dim (B^{n+1}_{c_0/\sqrt{10}}(0))$, by the transversality theorem given e.g. in \cite{Hirsch1994}*{Theorem 2.1}, there exists $\tilde m_j\colon U_j\to B^{n+1}_{c_0/\sqrt{10}}(0)$ such that $\tilde m_j(t)\not= 0$ and $\abs{\tilde m_j(t) - m_j(t)} < 2^{-j}$ for all $t\in U_j$. Hence, consider the function $a_j\colon U_j\to B^{n+1}_{2^{-j}}(0)$ given by $a_j(t) = m_j(t)-\tilde m_j(t)$ and note that $a_j(t)\not=m_j(t)$ for all $t\in U_j$. In particular, we can assume that there is $\eta_j>0$ such that $\abs{a_j(t) - m_j(t)}\ge \eta_j$ for all $t\in U_j$ (possibly taking $\delta$, and so $U_j$, smaller).

Now, for all $t\in U_j$, consider the one-parameter flow $\{\Phi^{t,j}(s,\cdot)\}_{s\ge 0} = \{\Phi^{\Sigma_t^j}(s,\cdot)\}_{s\ge 0}\subset \operatorname{Diff}(\overline{B}^{n+1})$ defined in \cref{lem:PushAway} and \[T_j = T(\eta_j,\delta, \avoids,\{F_v\}_{v\in\overline{B}^{n+1}}, c_0) \ge 0\] given by the lemma. Then, given a nonincreasing smooth function $\rho\colon [0,+\infty)\to [0,1]$ that is $1$ in $[0,3\delta/2]$ and $0$ in $[7\delta/4,+\infty)$, we define the continuous function
\[
v_j\colon I^n\to\overline{B}^{n+1}, \quad v_j(t) = \begin{cases}
\Phi^{t,j}(\rho(\vard(\avoids,\Sigma_t^j))T_j,\rho(\vard(\avoids,\Sigma_t^j))a_j(t)) & \text{for $t\in U_j$}\\
0 &\text{for $t\not\in U_j$},
\end{cases}
\]
and then set 
\[
\Lambda_t^j = \begin{cases}
                  F_{v_j(t)}(\Sigma_t^j)& \text{for $t\in U_j$}\\
\Sigma_t^j &\text{for $t\not\in U_j$}.
              \end{cases}
\]
Note that $\Lambda_t^j$ is $G$-equivariant since $\Sigma_t^j$ is $G$-equivariant and $F_{v_j(t)}\in \operatorname{Diff}_G(M)$.
However, a priori $\{\Lambda_t^j\}_{t\in I^n}$ is not contained in $\Pi$, since $(t,x)\mapsto F_{v_j(t)}(x)$ is not necessarily smooth but only continuous.
Anyway, let us first show that $\lim_{j\to+\infty}\sup_{t\in I^n} \area(\Lambda_t^j) \le W_\Pi$ and that $\Lambda_t^j \cap \overline{B}_\varepsilon^{\vard}(\avoids) =\emptyset$ for all $t\in I^n$, for $j$ sufficiently large, where $0<\varepsilon<\delta$ has to be chosen. Later, we will describe a regularization argument to get a sequence of sweepouts with the same properties of $\{\boldsymbol{\Lambda}^j\}_{j\in\N}$, but also contained in $\Pi$.

Observe that $\area(\Lambda_t^j) = \area(\Sigma_t^j)$ for $t\not\in U_j$ and, for $t\in U_j$, we have
\[
\area(\Lambda_t^j) = \area(F_{v_j(t)}(\Sigma_t^j)) \le \area(F_{\rho(\vard(\avoids,\Sigma_t^j)) a_j(t)}(\Sigma_t^j)).
\]
However $\abs{\rho(\vard(\avoids,\Sigma_t^j)) a_j(t)}\le 2^{-j}$, which implies that 
\[
\lim_{j\to+\infty} \max_{t\in I^n} \area(\Lambda_t^j) \le \lim_{j\to+\infty} \max_{t\in I^n} \area(\Sigma_t^j) = W_\Pi.
\]

Now let us prove that $\Lambda_t^j\cap\overline{B}_\varepsilon^{\vard}(\avoids) = \emptyset$ for all $t\in I^n$, for $j$ sufficiently large. 
Up to taking $\delta>0$ possibly smaller (note that \cref{lem:IndexDiffeo} still holds for $\delta$ smaller), we can assume that $\abs{\area(\tilde\avoids)-\area(\avoids)} \le c_0/20$ for all $\tilde\avoids\in \overline{B}_{2\delta}^{\vard}(\avoids)$.
Then, let us distinguish three cases:
\begin{itemize}
\item If $t\in I^n$ is such that $\vard(\avoids,\Sigma_t^j)\ge 7\delta/4$, then $\Lambda_t^j=\Sigma_t^j$ and therefore $\vard(\avoids,\Lambda_t^j) = \vard(\avoids,\Sigma_t^j) \ge 7\delta/4> \varepsilon$.

\item If $t\in I^n$ is such that $\vard(\avoids,\Sigma_t^j)\le 3\delta/2$, then we have $v_j(t) = \Phi^{t,j}(T_j,a_j(t))$ and therefore, by \cref{lem:PushAway}, it holds
\begin{align*}
\area(\Lambda_t^j) &= \area(F_{v_j(t)}(\Sigma_t^j)) = A^{\Sigma_t^j}(\Phi^{t,j}(T_j,a_j(t)))) < A^{\Sigma_t^j}(0) - \frac{c_0}{10} \\
&= \area(\Sigma_t^j) -\frac{c_0}{10} \le \area(\avoids) - \frac{c_0}{20},
\end{align*}
where the last inequality holds for $j$ sufficiently large.
Hence, it is possible to choose $\varepsilon>0$ possibly smaller (depending on $\avoids$ and $c_0$) such that this implies that $\vard(\avoids,\Lambda_t^j)>\varepsilon$ (indeed note that $c_0$ does not depend on $\varepsilon$).

\item If $t\in I^n$ is such that $3\delta/2\le \vard(\avoids,\Sigma_t^j)\le 7\delta/4$, then we apply \cref{lem:FarIsFar}.
Indeed, given $\overline{\eta} = \overline{\eta}(\delta, \avoids, \{F_v\}_{v\in \overline{B}^{n+1}})$ as in the lemma, for $j$ sufficiently large it holds that
\[
\area(\Lambda_t^j) = \area(F_{v_j(t)}(\Sigma_t^j)) \le \area(F_{\rho(\vard(\avoids,\Sigma_t^j) )a_j(t)}(\Sigma_t^j)) \le \area(\Sigma_t^j)+\overline{\eta},
\]
since $\abs{\rho(\vard(\avoids,\Sigma_t^j) )a_j(t)}\le 2^{-j}\to 0$. This implies that $\vard(\Lambda_t^j, \avoids)\ge 2\overline{\eta}$. Choosing $\varepsilon<2\overline{\eta}$, we then get that $\vard(\Lambda_t^j, \avoids)>\varepsilon$ for $j$ sufficiently large, as desired.
\end{itemize}

To conclude the proof, we need to address the regularity issue. For all $j\in\N$, let $\tilde v_j\colon I^n\to \overline{B}^{n+1}$ be a smooth function such that $\tilde v_j = 0$ on $I^n\setminus U_j$ and $\abs{\tilde v_j(t)-v_j(t)} \le 2^{-j}$ for all $t\in U_j$. Then, define
\[
\tilde\Lambda_t^j = \begin{cases}
                  F_{\tilde v_j(t)}(\Sigma_t^j)& \text{for $t\in U_j$}\\
\Sigma_t^j &\text{for $t\not\in U_j$}.
              \end{cases}
\]
Note that $\{\tilde\Lambda_t^j\}_{t\in I^n}\in \Pi$ for all $j\in\N$. Moreover, $\sup_{t\in I^n} \vard(\tilde\Lambda_t^j, \Lambda_t^j) \to 0$ as $j\to+\infty$, which implies that 
\[
\lim_{j\to+\infty} \sup_{t\in I^n}\area(\tilde\Lambda_t^j) = \lim_{j\to+\infty} \sup_{t\in I^n} \area(\Lambda_t^j) \le W_\Pi,
\]
and that $\tilde\Lambda_t^j\cap \overline{B}_\varepsilon^{\vard}(\avoids) = \emptyset$ for all $t\in I^n$, for $j$ sufficiently large (possibly taking $\varepsilon>0$ smaller).

Finally note that, thanks to \eqref{eq:AvoidK}, for all $t\in U_j$ it holds that $\tilde \Lambda_t^j = F_{\tilde v_j(t)}(\Sigma_t^j) \not \in K$, since $\Sigma_t^j\in \overline{B}^{\vard}_{2\delta}(\avoids)$ and $\tilde v_j(t)\in\overline{B}^{n+1}$. Moreover, for all $t\in I^n\setminus U_j$, we have $\tilde\Lambda_t^j = \Sigma_t^j\not\in K$. 
Hence $\{\tilde\Lambda_t^j\}_{t\in I^n}$ also avoids $K$ and thus satisfies the desired properties.
\end{proof}

\section{Proof of the main theorem}

We now have all the tools to prove \cref{thm:EquivMinMax}. Inspired by the proofs of \cite{MarquesNeves2016}*{Theorems 6.1 and 1.2}, the idea consists in repeatedly applying \cref{thm:Deformation} in order to obtain a minimizing sequence in $\Pi$ that stays away from the $G$-equivariant free boundary minimal surfaces with $G$-equivariant index greater than $n$.

\begin{proof}[Proof of \cref{thm:EquivMinMax}]
First of all, let us assume that the metric $\smetric$ on the ambient manifold $M$ is contained in $\mathcal{B}_G^\infty$, defined in \cref{def:BumpyMetrics}, i.e., it is bumpy.
Let us consider the set $\mathcal V^{n+1}$ of finite unions (possibly with multiplicity) of $G$-equivariant free boundary minimal surfaces in $M$ with area $W_\Pi$ and whose supports have $G$-equivariant index greater or equal than $n+1$. We want to prove that there exists a minimizing sequence $\{\so^j\}_{j\in\N}\subset \Pi$ such that $C(\{\so^j\}_{j\in\N})\cap \mathcal V^{n+1} = \emptyset$.
First note that, since $\smetric\in\mathcal{B}_G^\infty$, the set $\mathcal V^{n+1}$ is at most countable thanks to \cref{prop:CountableGequivSurfaces} (because $\mathcal{V}^{n+1}$ consists of finite unions with integer multiplicities of $G$-equivariant free boundary minimal surfaces). Therefore, we can write $\mathcal V^{n+1} = \{\avoids_1,\avoids_2,\ldots\}$. Now, the idea is to repeatedly apply \cref{thm:Deformation} in order to avoid all the elements in $\mathcal V^{n+1}$.

Let us consider a minimizing sequence $\{\so^j\}_{j\in\N}$ and apply \cref{thm:Deformation} with $\avoids=\avoids_1$. Then we get that there exist $\varepsilon_1>0$, $j_1\in\N$ and another minimizing sequence $\{\so^{1,j}\}_{j\in\N}\subset\Pi$ such that $\Sigma^{1,j}_t\cap \overline{B}^{\vard}_{\varepsilon_1}(\avoids_1)=\emptyset$ for all $j\ge j_1$ and $t\in I^n$. Moreover, we can assume that no $\avoids_k$ belongs to $\partial {B}^{\vard}_{\varepsilon_1}(\avoids_1)$.
Let us now consider $\avoids_2$: if it belongs to $\overline{B}^{\vard}_{\varepsilon_1}(\avoids_1)$, we choose $\varepsilon_2=\varepsilon_1-\vard(\avoids_1,\avoids_2)>0$ (here we use that $\avoids_2\not\in \partial {B}^{\vard}_{\varepsilon_1}(\avoids_1)$); otherwise we apply \cref{thm:Deformation} with $\avoids=\avoids_2$ and $K=\overline{B}^{\vard}_{\varepsilon_1}(\avoids_1)$. In both cases, we get $\varepsilon_2>0$, $j_2\in\N$ and another minimizing sequence $\{\so^{2,j}\}_{j\in\N}\subset\Pi$ such that $\Sigma^{2,j}_t\cap (\overline{B}^{\vard}_{\varepsilon_1}(\avoids_1)\cup \overline{B}^{\vard}_{\varepsilon_2}(\avoids_2))=\emptyset$ for all $j\ge j_2$ and $t\in I^n$. Moreover, we can assume again that no $\avoids_k$ belongs to $\partial {B}^{\vard}_{\varepsilon_2}(\avoids_2)$.

Then we proceed inductively for all $\avoids_k$'s and we have two possibilities:
\begin{itemize}
\item The process ends in finitely many steps. In this case there exist $m>0$, a minimizing sequence $\{\so^{m,j}\}_{j\in\N}\subset\Pi$, $\varepsilon_1,\ldots,\varepsilon_m>0$ and $j_m\in\N$ such that 
\[\Sigma^{m,j}_t\cap (\overline{B}^{\vard}_{\varepsilon_1}(\avoids_{1})\cup\ldots \cup \overline{B}^{\vard}_{\varepsilon_m}(\avoids_{m}))=\emptyset \]
for all $j\ge j_m$ and $t\in I^n$ and $\mathcal V^{n+1}\subset {B}^{\vard}_{\varepsilon_1}(\avoids_{1})\cup\ldots \cup {B}^{\vard}_{\varepsilon_m}(\avoids_{m})$.
\item The process continues indefinitely. In this case for all $m>0$ there exist a minimizing sequence $\{\so^{m,j}\}_{j\in\N}\subset\Pi$, $\varepsilon_m>0$ and $j_m\in\N$ such that $\Sigma^{m,j}_t\cap (\overline{B}^{\vard}_{\varepsilon_1}(\avoids_{1})\cup\ldots \cup \overline{B}^{\vard}_{\varepsilon_m}(\avoids_{m})) =\emptyset$ for all $j\ge j_m$ and $t\in I^n$ and no $\avoids_k$ belongs to $\partial {B}^{\vard}_{\varepsilon_1}(\avoids_{1})\cup\ldots \cup \partial {B}^{\vard}_{\varepsilon_m}(\avoids_{m})$.
\end{itemize}
In the first case we define $\boldsymbol{\Lambda}^i = \so^{m,i}$, while in the second case we set $\boldsymbol{\Lambda}^i= \so^{i,l_i}$ for all $i\in\N$, for some $l_i\ge j_i$ such that $\{\boldsymbol{\Lambda}^i\}_{i\in\N}\subset\Pi$ is a minimizing sequence and $C(\{\boldsymbol{\Lambda}^i\}_{i\in\N})\cap \mathcal V^{n+1}=\emptyset$. 
Hence, we can apply \cref{prop:ConvergenceToStationary}, \cref{lem:AlmostMinSeqInAdmissible,prop:AlmostMinSeq} to conclude the proof in the case of $\smetric\in\mathcal{B}_G^\infty$.

Now, consider the case of an arbitrary metric $\smetric$ and let $\{\smetric_k\}_{k\in\N}$ be a sequence of metrics in $\mathcal{B}_G^\infty$ converging smoothly to $\smetric$, which exists because of \cref{thm:BumpyIsGGeneric}. Thanks to the first part of the proof (in particular applying \cref{prop:ConvergenceToStationary,lem:AlmostMinSeqInAdmissible} to the minimizing sequence found in the first part of the proof with respect to $\gamma_k$), for every $k\in\N$ there exist a $G$-equivariant min-max sequence $\{\Lambda^{(k),j}_{t_j}\}_{j\in\N}\subset\Pi$ (i.e., $\area_{\gamma_k}(\Lambda^{(k),j}_{t_j}) \to W_{\Pi,\smetric_k}$, the width of $\Pi$ with respect to the metric $\smetric_k$) that is $G$-almost minimizing in every $L$-admissible family of $G$-equivariant annuli with $L=(3^n)^{3^n}$ and that converges to a finite union $\limitv_k$ of $G$-equivariant free boundary minimal (with respect to $\gamma_k$) surfaces (possibly with multiplicity). Moreover, it holds $\ind_G(\operatorname{spt} (\limitv_k))\le n$ and $\area_{\gamma_k}(\limitv_k) = W_{\Pi, \smetric_k}$.

Note that $W_{\Pi, \smetric_k}$ converges to the width $W_\Pi = W_{\Pi, \smetric}$ (the proof is the same as in the Almgren--Pitts setting, for which one can see \cite{IriMarNev18}*{Lemma 2.1}). Hence, since the varifolds $\limitv_k$ have uniformly bounded mass, up to subsequence $\limitv_k$ converges in the sense of varifolds to a varifold $\limitv$ with mass equal to $W_\Pi$. Moreover, taking a suitable diagonal subsequence of $\{\Lambda^{(k),j}_{t_j}\}_{j,k\in\N}$ we can obtain a min-max sequence $\{\Lambda^j\}_{j\in\N}$ for $\Pi$ with respect to $\smetric$, converging in the sense of varifolds to $\limitv$ and which is $G$-almost minimizing in every $L$-admissible family of $G$-equivariant annuli. Then, thanks to \cref{prop:AlmostMinSeq}, $\limitv$ is a disjoint union of $G$-equivariant free boundary minimal surfaces (possibly with multiplicity) and the genus bound in the statement holds.
One can look at \cref{fig:scheme} for a scheme of the argument.

\begin{figure}[htpb]
\centering
\usetikzlibrary{trees}
\tikzstyle{every node}=[anchor=west]
\tikzstyle{tit}=[shape=rectangle, rounded corners,
    draw, minimum width = 3cm]
\tikzstyle{sec}=[shape=rectangle, rounded corners]
\tikzstyle{optional}=[dashed,fill=gray!50]
\begin{tikzpicture}[scale=0.9]

\tikzstyle{mybox} = [draw, rectangle, rounded corners, inner sep=10pt, inner ysep=20pt]
\tikzstyle{fancytitle} =[fill=white, text=black, draw]

\begin{scope}[align=center, font=\small]
\node[sec] (d1) {$\vdots$};
\node[sec] (lk0) [below of = d1] {$\Lambda_{t_j}^{(k),j}$};
\node[sec] [below of = lk0] (lk1) {$\Lambda_{t_{j+1}}^{(k),j+1}$};
\node[sec] [below of = lk1] (lk2) {$\Lambda_{t_{j+2}}^{(k),j+2}$};
\node[sec] [below of = lk2] (d12) {$\vdots$};
\node[sec] [below of = d12, yshift=-.5cm] (x0) {$\Xi_k$};
\node[sec] [below of = x0, yshift=-.3cm] (s0) {$\operatorname{spt}(\Xi_k)$};
\draw[->] (d12) -- node {$\vard$} (x0);

\node[sec] [right of = d1, xshift=1cm] (d2) {$\vdots$};
\node[sec] (l10) [below of = d2] {$\Lambda_{t_j}^{(k+1),j}$};
\node[sec] [below of = l10] (l11) {$\Lambda_{t_{j+1}}^{(k+1),j+1}$};
\node[sec] [below of = l11] (l12) {$\Lambda_{t_{j+2}}^{(k+1),j+2}$};
\node[sec] [below of = l12] (d22) {$\vdots$};
\node[sec] [below of = d22, yshift=-.5cm] (x1) {$\Xi_{k+1}$};
\node[sec] [below of = x1, yshift=-.3cm] (s1) {$\operatorname{spt}(\Xi_{k+1})$};
\draw[->] (d22) -- node {$\vard$} (x1);

\node[sec] [right of = d2, xshift=1cm] (d3) {$\vdots$};
\node[sec] (l20) [below of = d3] {$\Lambda_{t_j}^{(k+2),j}$};
\node[sec] [below of = l20] (l21) {$\Lambda_{t_{j+1}}^{(k+2),j+1}$};
\node[sec] [below of = l21] (l22) {$\Lambda_{t_{j+2}}^{(k+2),j+2}$};
\node[sec] [below of = l22] (d32) {$\vdots$};
\node[sec] [below of = d32, yshift=-.5cm] (x2) {$\Xi_{k+2}$};
\node[sec] [below of = x2, yshift=-.3cm] (s2) {$\operatorname{spt}(\Xi_{k+2})$};
\draw[->] (d32) -- node {$\vard$} (x2);

\node[sec] [left of = x0, xshift = -.6cm] (dd1) {\ldots};
\node[sec] [below of = dd1, yshift=-.3cm] {\ldots};
\node[sec] [right of = x2, xshift = .6cm] (do) {\ldots};
\node[sec] [below of = do, yshift=-.3cm] (dos) {\ldots};

\node[sec]  [right of = do, xshift = 3cm] (lim) {$\Xi$};
\node[sec] [below of = lim, yshift=-.3cm] (limspt) {$\operatorname{spt}(\Xi)$};
\draw[->] ([xshift=3mm]do.east) -- node[above] {$\vard$} ([xshift=-3mm]lim.west);
\draw[->] ([xshift=3mm]dos.east) -- node[below] {smoothly in \\ $M\setminus(\mathcal{S}\cup\mathcal{Y})$\\
\tiny(multiplicity possible)} ([xshift=-3mm]limspt.west);

\node[sec] [right of = l21, xshift=.6cm, yshift=-1cm, rotate=9] (diag) {$\ddots$};

\begin{scope}[on background layer]
\draw[blue!50!white,rounded corners=15pt,opacity=1]
    let \p1=($(l10)!-32mm!(diag)$),
        \p2=($(diag)!-10mm!(l10)$),
        \p3=($(\p1)!6mm!90:(\p2)$),
        \p4=($(\p1)!6mm!-90:(\p2)$),
        \p5=($(\p2)!6mm!90:(\p1)$),
        \p6=($(\p2)!6mm!-90:(\p1)$)
    in
    (\p3) -- (\p4)-- (\p5) -- (\p6) -- cycle;
\end{scope}

\draw[->] ([xshift=1cm,yshift=-.3cm]diag.south) -- node[above] {$\vard$} ([xshift=-.3cm]lim.north);
\node[sec] [blue,above of = diag,xshift=3cm,yshift=5mm] (nn) {diagonal sequence,\\$G$-almost minimizing in\\ $L$-admissible families\\ of $G$-annuli};
\draw[->,blue] ([yshift=-2mm]nn.west) to [bend right=10] ([yshift=5mm,xshift=2mm]diag.north);

\draw [->, gray, line join=round, decorate, decoration={
    zigzag, segment length=4, amplitude=.9,post=lineto, post length=2pt
}]  ([yshift=-1mm]x0.south) -- ([yshift=1mm]s0.north);
\draw [->, gray, line join=round, decorate, decoration={
    zigzag, segment length=4, amplitude=.9,post=lineto, post length=2pt
}]  ([yshift=-1mm]x1.south) -- ([yshift=1mm]s1.north);
\draw [->, gray, line join=round, decorate, decoration={
    zigzag, segment length=4, amplitude=.9,post=lineto, post length=2pt
}]  ([yshift=-1mm]x2.south) -- ([yshift=1mm]s2.north);
\draw [->, gray, line join=round, decorate, decoration={
    zigzag, segment length=4, amplitude=.9,post=lineto, post length=2pt
}]  ([yshift=-1mm]lim.south) -- ([yshift=1mm]limspt.north);
\end{scope}

\end{tikzpicture}
\caption{Scheme of the convergence argument in the proof of \cref{thm:EquivMinMax}.}
\label{fig:scheme}
\end{figure}
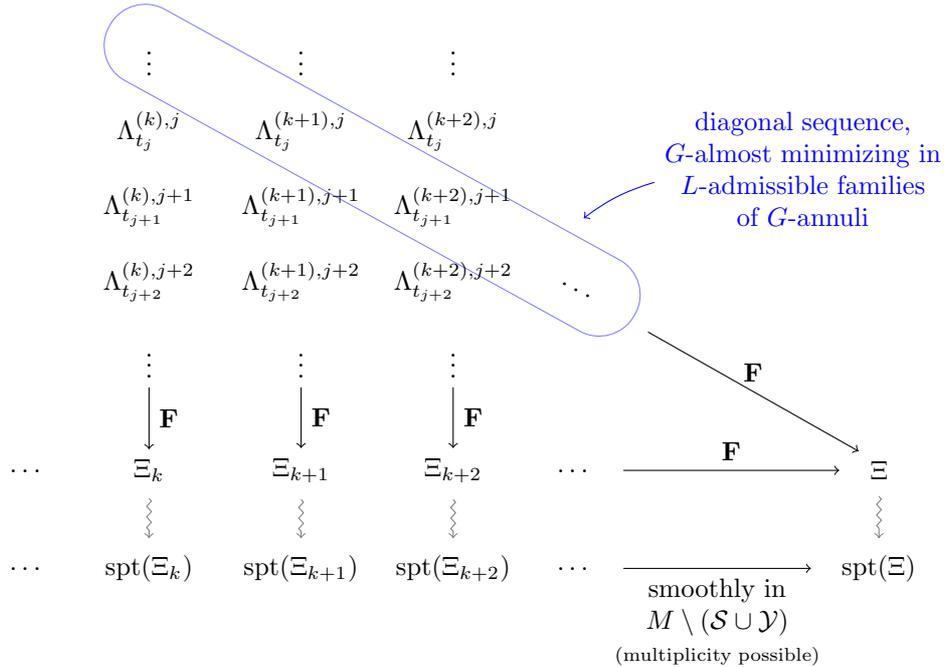

Now note that, thanks to \cref{thm:ConvBoundGIndex} (see also \cref{rem:AlsoForConvMetrics}), up to extracting a further subsequence, we can assume that $\operatorname{spt}(\limitv_k)$ converges smoothly (possibly with multiplicity) to a free boundary minimal surface away from the singular locus $\mathcal{S}$ and, possibly, from finitely many additional points $\mathcal{Y}$. This free boundary minimal surface coincides with the limit of $\operatorname{spt}(\limitv_k)$ in the sense of varifolds, which is $\operatorname{spt}(\limitv)$, in $M\setminus(\mathcal{S}\cup\mathcal{Y})$. In particular, we get that $\ind_G(\operatorname{spt}(\limitv))\le n$, which concludes the proof.
\end{proof}

\section{Equivariant index of some families of (free boundary) minimal surfaces} \label{sec:Applications}

In this section, we want to make use of \cref{thm:EquivMinMax} to compute the equivariant index of some families of minimal surface in $S^3$ and free boundary minimal surfaces in $B^3$.
Let us start by proving \cref{thm:Dg1Index}.

\begin{proof}[Proof of \cref{thm:Dg1Index}]
First note that, applying \cref{thm:EquivMinMax}, we can assume that the surface $M_g$ in \cite{CarlottoFranzSchulz2020}*{Theorem 1.1} has $\dih_{g+1}$-equivariant index less or equal than $1$. We want to prove that the $\dih_{g+1}$-equivariant index is exactly $1$. To do so, it is sufficient to construct a test function on which the associated quadratic form attains a negative value.

Recall that $M_g$ contains the horizontal axes of symmetry $\xi_1,\ldots,\xi_{g+1}$. Hence, $\operatorname{sgn}_{M_g}(\selem) = -1$ for all $\selem\in\dih_{g+1}$ given from a rotation of angle $\pi$ around any horizontal axis. Observe that these isometries generate all $\dih_{g+1}$, because the composition of the rotations of angle $\pi$ around $\xi_1$ and $\xi_2$ is equal to the rotation of angle $2\pi/(g+1)$ around $\xi_0$.
Hence, we can infer the sign of all the elements of the group.
In particular, we obtain that the function $u(x_1,x_2,x_3) = x_3$ on $M_g$ belongs to $C^\infty_G(M_g)$ for all $g\in \dih_{g+1}$. Moreover, one can compute that
\[
Q_{M_g}(u,u) = -\int_{M_g}\abs{A}^2u^2\de\Haus^2 <0
\]
(see \cite{Devyver2019}*{Lemma 6.1}). This proves that the $\dih_{g+1}$-equivariant index of $M_g$ is exactly $1$.
\end{proof}

\begin{remark}
Note that the function $x^\perp = \scal{x}{\nu}$ is in $C^\infty_G(M_g)$ and $Q(x^\perp, x^\perp) = 0$.
\end{remark}

\begin{remark}\label{rem:OtherFamilies}
Similarly to what we did for the surfaces $M_g$ constructed in \cite{CarlottoFranzSchulz2020}, we can apply \cref{thm:EquivMinMax} to any other surface obtained via an equivariant min-max procedure. To our knowledge, the known equivariant min-max constructions so far are:
\begin{itemize}
\item The minimal surfaces in $S^3$ of \cite{Ketover2016equiv}*{Sections 6.2, 6.3, 6.5, 6.6 and 6.7}.
\item The minimal surfaces in $S^3$ of \cite{KetoverMarquesNeves2020}*{Theorem 3.6}.
\item The free boundary minimal surfaces in $B^3$ of \cite{Ketover2016fb}*{Theorems 1.1, 1.2 and 1.3}.
\item The free boundary minimal surfaces in $B^3$ of \cite{CarlottoFranzSchulz2020}*{Theorem 1.1}, discussed above.
\end{itemize}

For all these surfaces we get that the equivariant index (with respect to the corresponding symmetry group) is less or equal than $1$. Then, one has to find a suitable equivariant test function to get the equality, on a case-by-case basis. However, in most of the cases above the constant function $1$ is such a test function. Indeed,  the constant $1$ is a negative direction for the second variation of the area functional for every minimal surface in $S^3$ and every free boundary minimal surface in $B^3$. Moreover, if the unit normal is equivariant (i.e., $h_*\nu=\nu$ for all $h$ in the symmetry group, where $\nu$ is a choice of unit normal), the constant function is also equivariant.

In other cases some more work is needed, as in the proof of \cref{thm:Dg1Index}. This last claim applies for example to the surfaces in $B^3$ constructed in \cite{Ketover2016fb}*{Theorem 1.1}, for which a good test function is again $u(x_1,x_2,x_3)=x_3$.
\end{remark}

\appendix

\section{Spectrum of elliptic operators with Robin boundary conditions}

In this appendix we consider an elliptic operator with Robin boundary conditions in presence of a symmetry group and we prove that it admits a discrete spectrum. The proof is very similar to the one in the case without equivariance (see for example the notes \cite{ArendtEtc2015}), but we report it here for completeness.

\begin{lemma} \label{lem:ImprovedTraceIneq}
Let $\Sigma^m$ be a compact Riemannian manifold with boundary $\partial \Sigma\not=\emptyset$. Then there exists a constant $C>0$ such that
\[
\norm{u}_{L^2(\partial\Sigma)}^2 \le C \norm{u}_{L^2(\Sigma)} \norm{u}_{H^1(\Sigma)}
\]
for all $u\in H^1(\Sigma)$.
\end{lemma}
\begin{proof}
The result is a slight variation of the standard trace inequality $\norm{u}_{L^2(\partial\Sigma)} \le C\norm{u}_{H^1(\Sigma)}$. Indeed, one can look at the proof of \cite{Brezis2011}*{Lemma 9.9} where $\Sigma=\R^m_+=\{(x_1,\ldots,x_m)\in\R^m\st x_m\ge 0\}$ and note that, integrating the second last line and applying Cauchy--Schwarz inequality, we get that there exists a constant $C>0$ such that
\[
\norm{u}^2_{L^2(\partial\Sigma)}  \le C \norm{u}_{L^2(\Sigma)} \norm{u}_{H^1(\Sigma)}
\]
for all $u\in C^1_c(\R^m)$. Then the proof of the lemma follows from a standard partition argument.
\end{proof}

\begin{lemma}\label{lem:CompactResolvent}
Let $V, H$ be Hilbert spaces such that there exists a compact (continuous) embedding $j\colon V\xhookrightarrow{d} H$ with dense image. Let $a\colon V\times V\to\R$ be a bounded symmetric $H$-elliptic form, i.e., assume that there exist $\omega\in\R$ and $c>0$ such that
\[
a(u,u) + \omega \norm{j(u)}^2_H \ge c \norm{u}^2_{V}
\]
for all $u\in V$. 
Moreover, let $A\colon D(A)\subset V\to H$ be the operator associated with the symmetric form $a$, i.e., given $x\in V$ and $y\in H$ we have $x\in D(A)$ and $Ax=y$ if and only if $a(x,u) = (y, j(u))_H$ for all $u\in V$. 
Then the operator $A+\omega \id \colon D(A)\subset V\to H$ is invertible with bounded compact inverse $(A+\omega\id)^{-1}\colon H\to D(A)\hookrightarrow H$.
\end{lemma}
\begin{proof}
We briefly sketch the proof.
Let us consider the operator $b\colon V\times V\to \R$ given by $b(u,v) \eqdef a(u,v) + \omega (j(u),j(v))_H$.
Since $b$ is bounded and coercive, we can apply Lax--Milgram theorem (cf. \cite{Brezis2011}*{Corollary 5.8}) and obtain that $\mathcal B\colon V\to V^*$ defined as $\mathcal Bu(v) \eqdef b(u,v)$ is an isomorphism. Finally one can prove that the operator $j\circ \mathcal B^{-1}\circ k$, where $k\colon H\to V^*$ is given by $k(y) = (y, j(\cdot))_H$, coincides with $(A+\omega\id)^{-1}$.
\end{proof}

\begin{theorem} \label{thm:DiscreteSpectrum}
Let $\Sigma^m$ be a compact Riemannian manifold with boundary $\partial\Sigma\not=\emptyset$, $G$ be a finite group of isometries of $\Sigma$ and $\operatorname{sgn}_\Sigma\colon G\to \{-1,1\}$ be a multiplicative function. In particular we can define $C^\infty_G(\Sigma), L^2_G(\Sigma), H^1_G(\Sigma)$ as in \cref{def:EquivFuncSpaces}. Let $\alpha\colon\Sigma\to\R$ and $\beta\colon\partial \Sigma\to\R$ be smooth $G$-equivariant functions, i.e., $\alpha\circ \selem = \alpha$ and $\beta\circ \selem = \beta$ for all $\selem\in G$.
Then there exists an orthonormal basis $(\varphi_k)_{k\ge 1}\subset C^\infty_G({\Sigma})$ of $L^2_G(\Sigma)$ and a nondecreasing sequence $(\lambda_k)_{k\ge 1}\subset \R$ diverging to $+\infty$ such that
\[
\begin{cases}
-\Delta \varphi_k - \alpha\varphi_k = \lambda_k \varphi_k & \text{in $\Sigma$}\\
\partial_\eta \varphi_k + \beta\varphi_k = 0 & \text{in $\partial\Sigma$}.
\end{cases}
\]
\end{theorem}

\begin{proof}
Let us consider the quadratic form $a\colon H^1_G(\Sigma)\times H^1_G(\Sigma)\to \R$ defined as 
\[
a(u,v) \eqdef \int_\Sigma (\grad u \cdot \grad v - \alpha u v ) \de \Haus^m + \int_{\partial \Sigma} \beta uv \de \Haus^{m-1}.
\]
Note that $a$ is continuous, because 
\begin{align*}
\abs{a(u,v)} &\le \norm{\grad u}_{L^2(\Sigma)} \norm{\grad v}_{L^2(\Sigma)} + \norm{\alpha}_{L^\infty(\Sigma)}\norm{u}_{L^2(\Sigma)} \norm{v}_{L^2(\Sigma)} + \norm{\beta}_{L^\infty(\partial \Sigma)}\norm{\tr u}_{L^2(\partial \Sigma)} \norm{\tr v}_{L^2(\partial \Sigma)},
\end{align*}
and the trace $\tr\colon H^1_G(\Sigma)\to L^2_G(\partial\Sigma)$ is continuous.
Moreover, $a$ is $L^2_G(\Sigma)$-elliptic, i.e.,
\[
a(u,u) +\omega \norm{u}_{L^2(\Sigma)}^2 \ge c \norm{u}^2_{H^1(\Sigma)}
\]
for all $u\in H^1_G(\Sigma)$, for some $c>0$ and $\omega\in\R$. Here we used that \cref{lem:ImprovedTraceIneq} implies
\[
\norm{\beta}_{L^\infty} \int_{\partial\Sigma} \abs{u}^2\de\Haus^{m-1} \le c_1 \norm{u}_{L^2(\Sigma)} \norm{u}_{H^1(\Sigma)} \le \frac 12\norm{u}_{H^1(\Sigma)}^2 + \frac {c_1}{2} \norm{u}_{L^2(\Sigma)}^2
\]
for some $c_1>0$. 

Hence, since the embedding $j\colon H^1_G(\Sigma)\to L^2_G(\Sigma)$ is compact, we can apply \cref{lem:CompactResolvent} with $H=L^2_G(\Sigma)$ and $V=H^1_G(\Sigma)$, and obtain that $A+\omega\id\colon D(A)\subset H^1_G(\Sigma)\to L^2_G(\Sigma)$ is invertible with bounded compact inverse $(A+\omega\id)^{-1}\colon L^2_G(\Sigma)\to D(A)\subset L^2_G(\Sigma)$, where $A\colon D(A)\subset L^2_G(\Sigma)\to L^2_G(\Sigma)$ is the operator associated with the symmetric form $a$, as in the statement of the lemma.

Now we want to prove that $A=-\Delta-\alpha$ and $D(A)=D_R$, where
\[
D_R\eqdef  \left\{u\in H^1_G(\Sigma)\st \Delta u \in L^2_G(\Sigma),\ \int_\Sigma (\Delta u) v + \grad u\cdot \grad v\de \Haus^m =- \int_{\partial \Sigma} \beta u v \de\Haus^{m-1}\ \forall v\in H^1_G(\Sigma) \right\}.
\]
Let $u\in D(A)$ and $f\in L^2_G(\Sigma)$ be such that $Au = f$, then
\[
\int_{\Sigma} (\grad u\cdot \grad v -\alpha uv) \de \Haus^m + \int_{\partial \Sigma} \beta uv \de \Haus^{m-1} = \int_\Sigma fv\de \Haus^m
\]
for all $v\in H^1_G(\Sigma)$. In particular, if we consider $v\in C^\infty_G(\Sigma)$ with compact support in $\operatorname{int}(\Sigma)$, we get that $\Delta u\in L^2_G(\Sigma)$ and $-\Delta u - \alpha u = f$. Exploiting this equation we obtain
\[
\int_\Sigma (\Delta u) v + \grad u\cdot \grad v\de \Haus^m =- \int_{\partial \Sigma} \beta u v \de\Haus^{m-1}
\]
for all $v\in H^1_G(\Sigma)$. This proves that $u\in D_R$ and $Au=-\Delta u-\alpha u$.
Conversely, given $u\in D_R$, we can easily show that $u\in D(A)$.

Now observe that the bounded compact operator $(A+\omega\id)^{-1}\colon L^2_G(\Sigma)\to L^2_G(\Sigma)$ is positive definite.
Therefore, by the spectral theorem, there exists an orthonormal basis $(\varphi_k)_{k\in\N}$ of $L^2_G(\Sigma)$ of eigenvectors of $(A+\omega\id)^{-1}$ with corresponding eigenvalues $(\mu_k)_{k\in\N}$ such that $\mu_0\ge \mu_1\ge\mu_2\ge \ldots\to0$. Then note that $\varphi_k\in D_R$ and it is an eigenvector of $A$ with corresponding eigenvalue $\lambda_k\eqdef 1/\mu_k-\omega$ for all $k\in\N$.
Finally, the fact that $\varphi_k\in C_G^\infty(\Sigma)$ follows from standard regularity theory and, from the fact that $\varphi_k\in D_R$, we obtain also that $\partial_\eta\varphi_k +\beta\varphi_k=0$ in $\partial\Sigma$.
\end{proof}

\section{Bumpyness is $G$-generic} \label{sec:BumpyIsGGeneric}

In this section, let $M^{m+1}$ be a smooth, compact, connected manifold with boundary and let $G$ be a finite group of diffeomorphisms of $M$.
Moreover, let $q$ be a positive integer $\ge 3$ or $q=\infty$.

\begin{definition} [{cf. \cite{AmbCarSha18Compactness}*{Theorem 9}}] \label{def:BumpyMetrics}
Denote by $\Gamma^q_G$ be the set of $G$-equivariant $C^q$ metrics on $M$ endowed with the $C^q$ topology. Moreover, let $\mathcal B^q_G\subset \Gamma^q_G$ be the subset of metrics $\smetric\in \Gamma^q_G$ such that no compact, smooth, $G$-equivariant manifolds with boundary that are $C^q$ properly embedded as free boundary minimal hypersurfaces in $(M,\smetric)$, and no finite covers of any such hypersurface, admit a nontrivial Jacobi field.
\end{definition}
\begin{remark}
Note that we \emph{do not} require the Jacobi field in the previous definition to be $G$-equivariant.
\end{remark}

\begin{proposition} \label{prop:CountableGequivSurfaces}
Let $\gamma\in \mathcal{B}_G^\infty$ be such that $(M,\smetric)$ satisfies property \hyperref[HypP]{$(\mathfrak{P})$}. Then $(M,\smetric)$ contains countably many $G$-equivariant free boundary minimal hypersurfaces.
\end{proposition}
\begin{proof}
This follows from Theorem 5 in \cite{AmbCarSha18Compactness}, similarly to Corollary 8 therein.
\end{proof}

\begin{theorem} \label{thm:BumpyIsGGeneric}
The subset $\mathcal B^q_G$ defined in \cref{def:BumpyMetrics} is comeagre in $\Gamma^q_G$.
\end{theorem}

\begin{remark}
In \cite{White2017}*{Theorem 2.1}, White proved a stronger result in the closed case, namely that a generic, $G$-equivariant, $C^q$ Riemannian metric on a smooth closed manifold is bumpy in the following sense: no closed, minimal \emph{immersed} submanifold has a nontrivial Jacobi field. We decided to state \cref{thm:BumpyIsGGeneric} only for properly embedded $G$-equivariant submanifolds, since White's generalization requires more technical work and we do not need it here. However a similar proof should work also in the case with boundary.
\end{remark}

Before proceeding to the proof of \cref{thm:BumpyIsGGeneric}, we need to introduce some notation, which are the adaptations to the $G$-equivariant setting to the ones given in \cite{AmbCarSha18Compactness}*{pp. 22}.
Consider a compact, connected, smooth manifold $\Sigma^m$ with boundary and fix any $\alpha\in (0,1)$. For $w\in C^{q-1,\alpha}(\Sigma,M)$, let
\[
[w]\eqdef \{w\circ\varphi \st \varphi\in \operatorname{Diff}(\Sigma)\}.
\]
Then define
\[
\mathcal{PE}^q_G \eqdef \{[w] \st w\in C^{q-1,\alpha}(\Sigma,M) \text{ is a proper embedding with $G$-equivariant image}\}
\]
and
\[
\mathcal{S}^q_G \eqdef \{(\gamma,[w]) \in \Gamma^q_G\times \mathcal{PE}^q_G \st w \text{ is a free boundary minimal proper embedding w.r.t. $\gamma$} \}.
\]
Finally denote by $\pi^q_G\colon \mathcal{S}^q_G\to \Gamma^q_G$ the projector onto the first factor, namely $\pi^q(\gamma,[w]) = \gamma$.

Given these definitions, we can now state the main ingredient to prove \cref{thm:BumpyIsGGeneric}. This is a structure theorem, whose first version was discovered by White in \cite{White1987-Bumpy}, stating the relation between critical points of $\pi^q$ and degenerate minimal hypersurfaces that underlies any bumpy metric theorem.

\begin{theorem} [Structure Theorem, cf. \cite{White2017}*{Theorem 2.3} and \cite{AmbCarSha18Compactness}*{Theorem 35}] \label{thm:StructureThm}
In the setting described above, $\mathcal S^q_G$ is a separable $C^1$ Banach manifold and $\pi^q_G\colon \mathcal{S}^q_G\to \Gamma^q_G$ is a $C^1$ Fredholm map of Fredholm index $0$.
Furthermore, $(\gamma, [w])$ is a critical point of $\pi^q_G$ if and only if $w(\Sigma)$ admits a nontrivial $G$-equivariant Jacobi field.
\end{theorem}

\begin{remark}
Theorem 2.3 in \cite{White2017} is the version of the theorem in the case without boundary, while Theorem 35 in \cite{AmbCarSha18Compactness} is the case with boundary but where no group of symmetries is considered.
\end{remark}

\begin{proof}
As observed by White in the proof of \cite{White2017}*{Theorem 2.3}, the proof of the Structure Theorem works the same in the equivariant case replacing ``metric'' with ``$G$-equivariant metric'', ``functions'' with ``$G$-equivariant functions'' and working on a $G$-equivariant hypersurface instead of on any hypersurface. In fact, the reader can follow the proof of the Structure Theorem in the case with boundary contained in \cite{AmbCarSha18Compactness}*{Section 7.2} with these modifications. In particular note that:
\begin{itemize}
\item Here we used $M^{m+1}$ for the ambient manifold and $\Sigma^m$ for the embedded hypersurface, instead of $\mathcal N^{n+1}$ and $M^n$ as in \cite{AmbCarSha18Compactness}.
\item The background metric $\gamma_*$ shall be chosen $G$-equivariant. This way, the exponential map with respect to $\gamma_*$ induces a diffeomorphism between $V^r$ (as defined in \cite{AmbCarSha18Compactness}*{pp. 25}) and a $G$-equivariant open neighborhood of $w(\Sigma)$. Then a class $[w]$ such that $w\in C^{q-1,\alpha}(\Sigma,M)$ is a proper embedding with $G$-equivariant image corresponds to a function $u\in C_G^{q-1,\alpha}(\Sigma,V)$ with $G$-equivariant symmetry (cf. \cref{def:EquivFuncSpaces}). 
\item As observed in \cref{rem:stationaryGstationary}, being $G$-stationary is the same as being stationary. Hence \cite{AmbCarSha18Compactness}*{Proposition 41} does not require modifications.
\item The computations in \cite{AmbCarSha18Compactness}*{Proposition 45} are the same. The only change that we need is in point (2), where the operator $L(\gamma,u)$ should be consider between the spaces $C^{q-1,\alpha}_G(\Sigma,V)$ and $C^{q-3,\alpha}_G(\Sigma,V)\times C^{q-2,\alpha}_G(\partial\Sigma,V)$. Then, similarly as in \cite{AmbCarSha18Compactness}, one can prove that $L(\gamma,u)\colon C^{q-1,\alpha}_G(\Sigma,V)\to C^{q-3,\alpha}_G(\Sigma,V)\times C^{q-2,\alpha}_G(\partial\Sigma,V)$ is a Fredholm operator of Fredholm index $0$.
\item Finally, \cite{AmbCarSha18Compactness}*{Proposition 46} and the final proof of Structure Theorem itself (at \cite{AmbCarSha18Compactness}*{pp. 31}) can be modified accordingly to fit the equivariant setting. Just note that the functions in $\ker L(\gamma, u)$, where $L(\gamma,u)$ is seen as a map between the equivariant spaces, correspond exactly to the $G$-equivariant Jacobi fields.  \qedhere
\end{itemize}
\end{proof}

\begin{proof}[Proof of \cref{thm:BumpyIsGGeneric}]
Given \cref{thm:StructureThm}, the proof is exactly the same as the proof of Theorem 9 in \cite{AmbCarSha18Compactness}*{Section 7.1} substituting $\Gamma^q$, $\mathcal S^q$, $\mathcal B^q$ with $\Gamma^q_G$, $\mathcal S^q_G$, $\mathcal B^q_G$. 
In fact, the only difference in the proofs is hidden in the application of \cite{White2017}*{Lemma 2.6}, which is the point where the Structure Theorem (in the different variants: \cref{thm:StructureThm}, \cite{White2017}*{Theorem 2.3} or \cite{AmbCarSha18Compactness}*{Theorem 35}) is used. However, White in \cite{White2017} works in the equivariant setting as we do here, hence the argument there can be applied in our context without changes, since the presence of the boundary does not play a role in this lemma, as observed also in \cite{AmbCarSha18Compactness}*{pp. 24}.
As said above (in the proof of \cref{thm:StructureThm}), the reader should pay attention to the fact that here we used $M^{m+1}$ for the ambient manifold and $\Sigma^m$ for the embedded hypersurface, instead of $\mathcal N^{n+1}$ and $M^n$ as in \cite{AmbCarSha18Compactness}.
\end{proof}

\bibliography{biblio}

\end{document}